\documentclass[a4paper, oneside, 11pt, final]{amsart}

\usepackage[utf8]{inputenc}
\usepackage[T1]{fontenc}
\usepackage[english]{babel}
\usepackage{lmodern}
\usepackage{xparse}
\usepackage{bbm}
\usepackage{setspace}
\usepackage[shortlabels]{enumitem}
\usepackage{romannum}
\usepackage[dvipsnames]{xcolor}
\usepackage[colorlinks, allcolors=Blue, final]{hyperref}
\usepackage[abbrev,backrefs]{amsrefs}
\usepackage{amssymb}
\usepackage{amsthm}
\usepackage{array} 
\usepackage{stmaryrd}
\usepackage{microtype}

\usepackage{geometry}
\usepackage{ifdraft}
\ifdraft{
	\geometry{textwidth=\textwidth,textheight=\textheight,paperwidth=\dimexpr\textwidth+1cm,paperheight=\dimexpr\textheight+2cm}}

\usepackage{thmtools}
\usepackage[capitalise]{cleveref}

\usepackage{mathtools}
\usepackage{newpxtext}
\usepackage{newpxmath}
\DeclareMathAlphabet{\mathbb}{U}{msb}{m}{n}

\newcommand*{\Diff}{\mathrm D}
\newcommand*{\dd}{\mathop{}\!{\operatorfont d}}

\newcommand*{\R}{\mathbb{R}} 
\newcommand*{\Hset}{\mathbb{H}}

\newcommand*{\Ball}{\mathbb B}

\DeclarePairedDelimiterX\norm[1]\lvert\rvert{
	\ifblank{#1}{\,\cdot\,}{#1}
}

\DeclarePairedDelimiterX\Norm[1]\lVert\rVert{
	\ifblank{#1}{\,\cdot\,}{#1}
}

\DeclarePairedDelimiterX{\floor}[1]{\lfloor}{\rfloor}{
	\ifblank{#1}{\,\cdot\,}{#1}
}

\DeclarePairedDelimiterX{\ceil}[1]{\lceil}{\rceil}{
	\ifblank{#1}{\,\cdot\,}{#1}
}

\providecommand{\st}{\,\vert\,}
\newcommand\stSymbol[1][]{%
	\nonscript\;#1\vert
	\allowbreak
	\nonscript\;
	\mathopen{}}
\DeclarePairedDelimiterX\set[1]\{\}{%
	\renewcommand\st{\stSymbol[\delimsize]}
	#1
}
\DeclarePairedDelimiter{\brk}{(}{)}

\DeclarePairedDelimiter{\abs}{\lvert}{\rvert}
\DeclarePairedDelimiterX{\intvc}[2]{[}{]}{#1,#2}
\DeclarePairedDelimiterX{\intvl}[2]{(}{]}{#1,#2}
\DeclarePairedDelimiterX{\intvr}[2]{[}{)}{#1,#2}
\DeclarePairedDelimiterX{\intvo}[2]{(}{)}{#1,#2}

\theoremstyle{plain}
\declaretheorem[name=Theorem, numberwithin=section]{theorem}

\declaretheorem[name=Lemma, sibling=theorem]{lemma}
\declaretheorem[name=Proposition, sibling=theorem]{proposition}

\theoremstyle{definition}

\declaretheorem[name=Remark, sibling=theorem,qed={$\triangle$}]{remark}

\renewcommand{\PrintDOI}[1]{%
	\href{http://dx.doi.org/#1}{doi:#1}%
}
\BibSpec{book}{%
	+{}  {\PrintPrimary}                {transition}
	+{,} { \textit}                     {title}
	+{.} { }                            {part}
	+{:} { \textit}                     {subtitle}
	+{,} { \PrintEdition}               {edition}
	+{}  { \PrintEditorsB}              {editor}
	+{,} { \PrintTranslatorsC}          {translator}
	+{,} { \PrintContributions}         {contribution}
	+{,} { }                            {series}
	+{,} { \voltext}                    {volume}
	+{,} { }                            {publisher}
	+{,} { }                            {organization}
	+{,} { }                            {address} 
	+{,} { \PrintDateB}                 {date}
	+{,} { }                            {status}
	+{,} { \PrintDOI}                   {doi}
	+{}  { \parenthesize}               {language}
	+{}  { \PrintTranslation}           {translation}
	+{;} { \PrintReprint}               {reprint}
	+{.} { }                            {note}
	+{.} {}                             {transition}
	+{}  {\SentenceSpace \PrintReviews} {review}
}

\title{Morrey--Sobolev inequalities with power weights on the half-space}

\address{
	Universit\'e catholique de Louvain,
	Institut de Recherche en Math\'ematique et Physique,
	Chemin du Cyclotron 2 bte LJ.01.01,
	1348 Louvain-la-Neuve, Belgium
	}

\author[J.\ Van Schaftingen]{Jean Van Schaftingen}
\author[L.\ Winter]{Leon Winter}
\email{Jean.VanSchaftingen@uclouvain.be,
	Leon.Winter@uclouvain.be
	}

\thanks{Both authors were supported by the Fonds Spéciaux de Recherche (FSR), UClouvain. Jean Van Schaftingen was supported by the Projet de Recherche T.0229.21 ``Singular Harmonic Maps and Asymptotics of Ginzburg--Landau Relaxations'' of the Fonds de la Recherche Scientifique--FNRS and Leon Winter is a research fellow of the  Fonds de la Recherche Scientifique--FNRS}


\begin{document}
	\pagenumbering{arabic}
\begin{abstract}
	Morrey--Sobolev inequalities are established for functions in weighted Sobolev spaces on the \(n\)-dimensional half-space, where the weight is a power of the distance to the boundary, as well as for Sobolev spaces on the \(n\)-dimensional hyperbolic space.
	All the estimates are optimal up to a multiplicative constant.
\end{abstract}
	\maketitle

\section{Introduction}

The classical \emph{Morrey--Sobolev inequality}, due to Morrey~\citelist{\cite{Morrey_1938}\cite{Morrey_1966}}, states that, if \(p > n\), there exists a constant \(C > 0\) such that every Sobolev function
\begin{equation}
\label{eq_quai5niex2teiLuc3pee1weo}
 u \in \smash{\dot{W}}^{1,p} \brk{\R^n}
 \coloneqq
  \set[\bigg]{u \in W^{1,1}_{\mathrm{loc}}\brk{\R^n} \st \int_{\R^n} \norm{\Diff u}^p < +\infty}
\end{equation}
satisfies,
for almost every \(x\), \(y \in \R^n\),
\begin{equation}
\label{eq_izaitheungae7ubahj4eeSah}
 \abs{u \brk{x} - u \brk{y}}
 \le C \norm{x -y}^{1 - \frac{n}{p}}
 \brk[\bigg]{\int_{\R^n} \norm{\Diff u}^p}^\frac{1}{p}
\end{equation}
(see also~\citelist{\cite{Brezis_2011}*{Thm.\ 9.12}\cite{Willem_2022}*{Thm.\ 6.4.4}\cite{Mazya_2011}*{Thm.\ 1.4.5}}).
In other words, the Sobolev space \(\smash{\dot{W}}^{1, p}\brk{\R^n}\) is continuously embedded in the space \(C^{0, 1 - \frac{n}{p}} \brk{\R^n}\) of H\"older-continuous functions of order \(1 - \frac{n}{p}\).

We consider, for \(\gamma \in \R\), the \emph{weighted Sobolev space}
\[
 u \in \smash{\dot{W}}^{1,p}_\gamma\brk{\R^n_+}
 \coloneqq
  \set[\bigg]{u \in W^{1,1}_{\mathrm{loc}}\brk{\R^n_+} \st \int_{\R^n_+} \norm{\Diff u \brk{z}}^p {z_n}^\gamma \dd z < +\infty}.
\]
The fundamental properties of weighted Sobolev spaces have been studied for wide classe of weights \citelist{\cite{Kufner_1980}\cite{Turesson_2000}}.
For \(p < n\), Maz'ya has proved scale invariant Sobolev embeddings  for compactly supported functions with power weights \cite{Mazya_2011}*{Cor.\ 2.1.7/2}.

In the context of \(\smash{\dot{W}}^{1,p}_\gamma\brk{\R^n_+}\), a local version of \eqref{eq_izaitheungae7ubahj4eeSah} shows that, for almost every \(x\), \(y \in \R^n_+ \coloneqq  \R^n \times \intvo{0}{+\infty}\) satisfying \(x_n \le y_n \le 2 x_n\) and \(\abs{x - y} \le 2x_n\),
\begin{equation}
\label{eq_eenoqu0oosah7uNgaiThohza}
 \abs{u \brk{x} - u \brk{y}}
 \le C \frac{\norm{x -y}^{1 - \frac{n}{p}}}{{x_n}^\frac{\gamma}{p}}
 \brk[\bigg]{\int_{\R^n_+} \norm{\Diff u \brk{z}}^p {z_n}^\gamma \dd z }^\frac{1}{p}.
\end{equation}
When \(-1 < \gamma < p - 1\), setting \(s \coloneqq 1 - \smash{\frac{\gamma + 1}{p}}\) so that \(\gamma = \brk{1 - s}p - 1\), the classical trace theory, due to Gagliardo~\cite{Gagliardo_1957} and Uspenski\u{\i}~\cite{Uspenskii_1961} (see also~\citelist{\cite{Mironescu_Russ_2015}\cite{Leoni_fractionnal_2023}*{Thm.\ 9.4}}), states that any function \(u \in \smash{\smash{\dot{W}}^{1,p}_\gamma}\brk{\R^n_+}\) has a well-defined trace \(\operatorname{tr}_{\R^{n - 1}} u \in \smash{\dot{W}}^{s,p}\brk{\R^{n - 1}}\) on \(\R^{n - 1} \cong \R^{n - 1} \times \set{0}\ = \partial\R^n_+\), which satisfies the estimate
\begin{equation}
\label{eq_aequi9aelePea4fe5noht6ya}
 \smashoperator[r]{\iint_{\R^{n - 1}\times \R^{n - 1}}} \frac{\abs{\operatorname{tr}_{\R^{n - 1}} u \brk{x} - \operatorname{tr}_{\R^{n - 1}} u \brk{y}}^p}{\norm{x - y}^{n - 1 + sp}}\dd x \dd y
 \le C
 \int_{\R^n_+} \norm{\Diff u \brk{z}}^p {z_n}^\gamma \dd z .
\end{equation}
On the other hand, the Morrey--Sobolev inequality in fractional Sobolev spaces~\citelist{\cite{Taibleson_1964}*{Lem.\ 11}\cite{DiNezza_Palatucci_Valdinoci_2012}*{\S 8}\cite{Leoni_fractionnal_2023}*{Thm.\ 7.13}} states that, if \(v \in \smash{\dot{W}}^{s,p}\brk{\R^{n - 1}}\) and \(sp > n - 1\), then, for almost every \(x, y \in \R^{n - 1}\),
\begin{equation}
\label{eq_EM0aedu5reihi5ooxieQu3qu}
  \abs{v \brk{x} - v \brk{y}}
  \le C \norm{x - y}^{s - \frac{n - 1}{p}} \brk[\bigg]{
   \smashoperator[r]{\iint_{\R^{n - 1}\times \R^{n - 1}}} \frac{\abs{v \brk{z} - v \brk{w}}^p}{\norm{z - w}^{n - 1 + sp}}\dd z \dd w}^{\frac{1}{p}}.
\end{equation}
If \(\gamma < p - n\), combining \eqref{eq_aequi9aelePea4fe5noht6ya} and \eqref{eq_EM0aedu5reihi5ooxieQu3qu} with \(s = 1 - \frac{\gamma + 1}{p}\), we get
\begin{equation}
\label{eq_EeNgaiY7see3eifaim1xeew6}
  \abs{\operatorname{tr}_{\R^{n - 1}}u \brk{x} - \operatorname{tr}_{\R^{n - 1}}u \brk{y}}
  \le C \norm{x - y}^{1 - \frac{n + \gamma}{p}}
 \brk[\bigg]{ \int_{\R^n_+} \norm{\Diff u \brk{z}}^p {z_n}^\gamma \dd z}^{\frac{1}{p}}.
\end{equation}

The starting point of the present work is to understand whether the continuity of \(u \in \smash{\dot{W}}^{s,p}\brk{\R^{n}_+}\) in the interior of \(\R^n_+\) and of its trace on the boundary \(\partial \R^n_+ \cong \R^{n - 1}\), that is, the estimates \eqref{eq_eenoqu0oosah7uNgaiThohza} and \eqref{eq_EeNgaiY7see3eifaim1xeew6}, both follow from a single global continuity result, and what the corresponding estimates would be.
This question is answered by the following result.

\begin{theorem}
	\label{thm: weighted Morrey--Sobolev}
	Let  \(n\geq 1\), \(\gamma \in \R\), and \(n < p < +\infty\).
	There exists constant \(C = C\brk{n,\gamma,p} > 0\), depending at most on \(n\), \(\gamma\), and \(p\),
	such that every \(u \in  \smash{\smash{\dot{W}}^{1,p}_{\smash{\gamma}}}\brk{\R_+^n}\)
	satisfies for almost every \(x\), \(y \in \R^n_+\),
	\[
	  \abs{u \brk{x} - u\brk{y}}
 \le C\, \Theta_{\smash{1-\frac{1}{p}, 1-\frac{n}{p}, 1 - \frac{n + \gamma}{p}}}\brk{x, y}\,
 \brk[\bigg]{\int_{\R^n_+} \norm{\Diff u\brk{z}}^p {z_n}^\gamma \dd z}^\frac{1}{p}.
	\]
\end{theorem}

Here and in the sequel, we use the function
\(\Theta_{\alpha, \beta, \kappa} \colon \R^n_+ \times \R^n_+ \to \intvr{0}{+\infty}\) defined for given \(\alpha > 0\), \(\beta \in \R\), and \(\gamma \in \R\), by
\begin{equation}
\label{eq_hoh1az5ishahXeera3Daepod}
 \Theta_{\alpha, \beta, \kappa} \brk{x, y}
 =
 \left\{
 \begin{aligned}
		&\frac{\abs{x - y}^{\beta}}{\min \brk{x_n, y_n}^{-\kappa}\max\brk{x_n, y_n, \abs{x - y}}^{\beta} }
		& &\text{if \(\kappa < 0\),}\\
	   &\brk[\bigg]{\ln \brk[\Big]{1 + \brk[\Big]{\frac{\abs{x - y}}{\min \brk{x_n, y_n}}}^\frac{\beta}{\alpha}}}^{\alpha}
	   &&\text{if \(\kappa = 0\)},\\
	   & \frac{\norm{x-y}^{\beta}}{\max \brk{x_n, y_n,
				\norm{x - y}}^{\beta - \kappa}}
				 & & \text{if \(0 < \kappa < 1 + \beta - \alpha\),}\\
		&\frac{\max\brk{x_n, y_n}^{\kappa} \abs{x - y}^{\beta}}{\max \brk{x_n, y_n, \abs{x - y} }^{\beta}} && \text{if \(\kappa \ge 1 + \beta - \alpha\)}.
 \end{aligned}
 \right.
\end{equation}

Equivalently, \cref{thm: weighted Morrey--Sobolev} states that we have
	\begin{itemize}
		\item
		if \(\gamma > p - n\),
		\begin{equation}
		\label{thm: weighted Morrey--Sobolev gamma > p - n}
			\norm{u\brk{x} - u \brk{y}}
			\leq C \frac{\abs{x - y}^{1 - \frac{n}{p}}}{\min \brk{x_n, y_n}^{\frac{n + \gamma}{p} - 1} \max\brk{x_n, y_n, \abs{x - y}}^{1 - \frac{n}{p}} }
			 \brk[\bigg]{\int_{\R^n_+}
	 \norm{\Diff u\brk{z}}^p {z_n}^\gamma \dd z}^\frac{1}{p},
		\end{equation}
		\item
		if \(\gamma = p - n\),
		\begin{equation}
		\label{thm: weighted Morrey--Sobolev gamma = p - n}
				\norm{u\brk{x} - u\brk{y}}
				\leq C
				\brk[\bigg]{\ln \brk[\Big]{1 + \brk[\Big]{\frac{\abs{x - y}}{\min \brk{x_n, y_n}}}^\frac{p - n}{p - 1}}}^{1 - \frac{1}{p}}
				 \brk[\bigg]{\int_{\R^n_+}
	 \norm{\Diff u\brk{z}}^p {z_n}^\gamma \dd z}^\frac{1}{p},
		\end{equation}
		\item
		if \(-1 \le  \gamma < p - n\),
		\begin{equation}
		\label{thm: weighted Morrey--Sobolev gamma < p - n}
			\norm{u\brk{x} - u \brk{y}}
			\leq C
			\frac{\norm{x-y}^{1 - \frac{n}{p}}}{\max \brk{x_n, y_n,
				\norm{x - y}}^{ \frac{\gamma}{p}}}
				 \brk[\bigg]{\int_{\R^n_+}
	 \norm{\Diff u\brk{z}}^p {z_n}^\gamma \dd z}^\frac{1}{p},
		\end{equation}
		\item
		if \(\gamma \leq -1\),
		\begin{equation}
		\label{thm: weighted Morrey--Sobolev gamma <= -1}
			\norm{u\brk{x} - u \brk{y}}
			\leq C
			\frac{\max\brk{x_n, y_n}^{1 - \frac{n + \gamma}{p}} \abs{x - y}^{1 - \frac{n}{p} }}{\max \brk{x_n, y_n, \abs{x - y} }^{1 - \frac{n}{p}}}
			\brk[\bigg]{\int_{\R^n_+} \norm{\Diff u\brk{z}}^p {z_n}^\gamma \dd z}^\frac{1}{p}.
		\end{equation}
	\end{itemize}

For any fixed \(n\), \(\gamma\), and \(p\), it can be observed that the inequality among \eqref{thm: weighted Morrey--Sobolev gamma > p - n}-\eqref{thm: weighted Morrey--Sobolev gamma <= -1} which applies is always stronger than the preceding ones in the list.

It can be seen that the interior estimate \eqref{eq_eenoqu0oosah7uNgaiThohza} and the boundary estimate \eqref{eq_EeNgaiY7see3eifaim1xeew6} are both consequences of \cref{thm: weighted Morrey--Sobolev} with \(\gamma < p - n\) (see \eqref{thm: weighted Morrey--Sobolev gamma < p - n}).
Conversely, the classical extension theory for weighted Sobolev spaces \eqref{eq_aequi9aelePea4fe5noht6ya} can be combined with \eqref{eq_EeNgaiY7see3eifaim1xeew6} to recover the fractional Morrey--Sobolev embedding \eqref{eq_EM0aedu5reihi5ooxieQu3qu}.
\begin{remark}
\Cref{thm: weighted Morrey--Sobolev} implies that, if
\(\abs{x - y} \le \frac12 \min \brk{x_n, y_n}\),
then
\begin{equation}
\label{eq_xichoh6pee7iechaefohxu4W}
\abs{u \brk{x} - u \brk{y}}
\le C \frac{\norm{x-y}^{1 - \frac{n}{p}}}{\min \brk{{x_n}^{\frac{\gamma}{p}}, {y_n}^{\frac{\gamma}{p}}}}
				 \brk[\bigg]{\int_{\R^n_+}
	 \norm{\Diff u\brk{z}}^p {z_n}^\gamma \dd z}^\frac{1}{p}.
\end{equation}
Indeed, one has then
\[
  \min \brk{x_n, y_n} \le \max \brk{x_n, y_n} \le \min \brk{x_n, y_n} + \abs{x - y} \le \tfrac{3}{2} \min \brk{x_n, y_n},
\]
so that
\[
 \min \brk{x_n, y_n}
 \le
 \max \brk{x_n, y_n, \abs{x - y}} \le \tfrac{3}{2} \min \brk{x_n, y_n}.\qedhere
\]
\end{remark}
\begin{remark}
When \(\gamma \leq 0\), the estimate
\eqref{eq_xichoh6pee7iechaefohxu4W} holds for almost every \(x\), \(y \in \R^{n - 1}\cong\partial \R^n_+\) and for almost every \(x \in \R^{n - 1}\) and almost every \(y \in \R^n_+\), provided \(u\) is replaced by its trace on the boundary.
\end{remark}

\begin{remark}
As for the classical Morrey--Sobolev embedding, \cref{thm: weighted Morrey--Sobolev} implies that any function \(u \in \smash{\smash{\dot{W}}^{1, p}_{\smash{\gamma}}\brk{\R^n_+}}\) is equal almost everywhere to a unique continuous function \(u_\ast \in C \brk{\R^n_+}\).
When \(\gamma < p - n\), this function \(u_\ast\) can be taken to be \emph{continuous up to the boundary} \(\partial \R^n_+ = \R^{n - 1} \times \set{0}\), and when \(\gamma \le -1\), to be \emph{constant on the boundary}.
\end{remark}

\begin{remark}
When \(0 \le \gamma < p - n\),
\cref{thm: weighted Morrey--Sobolev} implies through \eqref{thm: weighted Morrey--Sobolev gamma < p - n} that, for almost every \(x\), \(y \in \R^n_+\),
\begin{equation}
\label{eq_thushePhie5doo8eixei2vei}
  \abs{u \brk{x} - u \brk{y}}
	 \le
	 C
	 \abs{x - y}^{1 - \frac{n + \gamma}{p}}
	 \brk[\bigg]{\int_{\R^n_+}
	 \norm{\Diff u\brk{z}}^p {z_n}^\gamma \dd z}^\frac{1}{p},
\end{equation}
since one has then
\(
\abs{x - y}^\frac{\gamma}{p}
\le
 \max \brk{x_n, y_n, \abs{x - y}}^\frac{\gamma}{p}.
\)
The global H\"older exponent is thus smaller than the one appearing in the classical Morrey--Sobolev inequality \eqref{eq_izaitheungae7ubahj4eeSah}.
The estimate \eqref{eq_thushePhie5doo8eixei2vei} is a particular case of Cabré and Ros-Oton’s estimate for monomial weights~\cite{Cabre_RosOton_2013}.
\end{remark}

\begin{remark}
When \(-1 < \gamma \le 0\), \cref{thm: weighted Morrey--Sobolev} and \eqref{thm: weighted Morrey--Sobolev gamma < p - n} can be rewritten as
\[
  \abs{u \brk{x} - u \brk{y}}
	 \le
	 C
	 \brk[\big]{\abs{x - y}^{1 - \frac{n + \gamma}{p}} + \min \brk{x_n, y_n}^{-\frac{\gamma}{p}}\abs{x - y}^{1 - \frac{n}{p}}}
	 \brk[\bigg]{\int_{\R^n_+}
	 \norm{\Diff u\brk{z}}^p {z_n}^\gamma \dd z}^\frac{1}{p},
\]
and can be seen as an improvement close to the boundary of the interior H\"older-continuity exponent \(1 - \frac{n}{p}\) towards the boundary exporent \(1 - \frac{n + \gamma}{p} > 1 - \frac{n}{p}\).
\end{remark}

\begin{remark} \label{rmk: Berestycki_Lions}
 When \(\gamma > p - n\) and if for every \(R > 0\),
\(\operatorname{ess\ inf}_{\R^n_+ \setminus \Ball^n_R}\, \abs{u} = 0\), then, for almost every \(x \in \R^n_+\), letting \(y_n \to +\infty\) in \eqref{thm: weighted Morrey--Sobolev gamma > p - n}, we get the estimate
\begin{equation}
\label{ineq_strauss}
 \abs{u \brk{x}}
 \le  \frac{C}{{x_n}^{\frac{n + \gamma}{p} - 1}}
 \brk[\bigg]{\int_{\R^n_+}
	 \norm{\Diff u\brk{z}}^p {z_n}^\gamma \dd z}^\frac{1}{p}.
\end{equation}
When \(n = 1\), the latter inequality \eqref{ineq_strauss} corresponds to Berestycki and Lions’ radial lemma
\citelist{\cite{Berestycki_Lions_1983}*{Lemma A.III}\cite{Lions_1982}*{Corollaire II.1}}.
\end{remark}

\begin{remark}
When \(n = 1\), one has \(\abs{x - y} \le \max \brk{x, y}\) on \(\R_+ = \brk{0, +\infty}\).
Therefore, \(\max \brk{x, y, \abs{x - y}} = \max \brk{x, y}\) and \(\min \brk{x, y} + \abs{x - y} = \max \brk{x, y}\).
It follows then from \cref{thm: weighted Morrey--Sobolev} that
\begin{itemize}
 \item if \(\gamma > p - 1\),
		\begin{equation}
		\label{eq_mahW1OteGee6eek7Ohgoofei}
			\norm{u\brk{x} - u \brk{y}}
			\leq C \frac{\abs{x - y}^{1 - \frac{1}{p}}}{\min \brk{x, y}^{\frac{1 + \gamma}{p} - 1} \max\brk{x, y}^{1 - \frac{1}{p}} }
			 \brk[\bigg]{\int_{\R_+}
	 \norm{u'\brk{z}}^p z^\gamma \dd z}^\frac{1}{p},
		\end{equation}
\item
if \(\gamma = p - 1\),
		\begin{equation}
		\label{eq_ou8ahb2hai0Moh4aip5uNgoo}
				\norm{u\brk{x} - u\brk{y}}
				\leq C\,
				\abs[\Big]{\ln \frac{x}{y}
				}^{1 - \frac{1}{p}}
				 \brk[\bigg]{\int_{\R_+}
	 \norm{u'\brk{z}}^p z^\gamma \dd z}^\frac{1}{p},
		\end{equation}
\item
if \(\gamma < p - 1\),
		\begin{equation}
\label{it_aiyokaighahtheesh4maiV8j}
			\norm{u\brk{x} - u \brk{y}}
			\leq C
			\frac{\norm{x-y}^{1 - \frac{1}{p}}}{\max \brk{x, y}^{ \frac{\gamma}{p}}}
				 \brk[\bigg]{\int_{\R_+}
	 \norm{u' \brk{z}}^p z^\gamma \dd z}^\frac{1}{p}.
		\end{equation}
\end{itemize}
Notably, there is no difference between the cases \(\gamma \le -1\) and \(-1 \le \gamma < p - 1\) that are merged into a single case \eqref{it_aiyokaighahtheesh4maiV8j} from the two cases \eqref{thm: weighted Morrey--Sobolev gamma < p - n} and \eqref{thm: weighted Morrey--Sobolev gamma <= -1} in \cref{thm: weighted Morrey--Sobolev}.

The estimates \eqref{eq_mahW1OteGee6eek7Ohgoofei}, \eqref{eq_ou8ahb2hai0Moh4aip5uNgoo} and \eqref{it_aiyokaighahtheesh4maiV8j} for \(n = 1\) can all be obtained by a direct application of the fundamental theorem of calculus and H\"older’s inequality:
\[
 \abs{u \brk{x} - u \brk{y}}
 \le \int_{\intvc{x}{y}} \abs{u'\brk{z}} \dd z
 \le \brk[\bigg]{\int_{\R_+} \abs{u' \brk{z}}^p z^\gamma \dd z}^\frac{1}{p}
 \brk[\bigg]{\int_{\intvc{x}{y}} \frac{1}{z^\frac{\gamma}{p - 1}} \dd z}^{1 - \frac{1}{p}}.\qedhere
\]
\end{remark}
The proof of \eqref{thm: weighted Morrey--Sobolev gamma > p - n}, \eqref{thm: weighted Morrey--Sobolev gamma = p - n}, and \eqref{thm: weighted Morrey--Sobolev gamma < p - n} in \cref{thm: weighted Morrey--Sobolev} is based on the Sobolev representation formula, whereas the proof of \eqref{thm: weighted Morrey--Sobolev gamma <= -1} is obtained as an enhancement of \eqref{thm: weighted Morrey--Sobolev gamma < p - n} through the observation that traces of Sobolev functions in \(\smash{\dot{W}}^{1,p}_{\smash{\gamma}} \brk{\R^n_+}\) with \(\gamma \leq -1\) are constant.

We also show that the result of \cref{thm: weighted Morrey--Sobolev} is \emph{optimal up to a multiplicative constant}.
\begin{theorem}
\label{theorem_optimal}
	Let  \(n\geq 1\), \(\gamma \in \R\), and \(n < p < +\infty\).
	There exists a constant \(C > 0\) such that given \(\omega \in C \brk{\R^n_+ \times \R^n_+, \intvr{0}{+\infty}}\) for which every \(u \in \smash{\smash{\dot{W}}^{1,p}_{\smash{\gamma}}}\brk{\R_+^n}\) satisfies for almost every \(x\), \(y \in \R^n_+\), the estimate
	\begin{equation}
	\label{eq: hypothesis on omega}
	 \abs{u \brk{x} - u \brk{y}}
	 \le
	 \omega \brk{x, y}\,
	 \brk[\bigg]{\int_{\R^n_+}
	 \norm{\Diff u\brk{z}}^p {z_n}^\gamma \dd z}^\frac{1}{p},
	\end{equation}
	then necessarily, for all \(x\), \(y \in \R^n_+\),
	\[
	   \omega \brk{x, y} \ge C\, \Theta_{1 - \frac{1}{p},1 - \frac{n}{p}, 1 - \frac{n + \gamma}{p}}
	   \brk{x, y}.
	\]
\end{theorem}

The determination of the optimal \(\omega\) is a delicate question. We show that such an \(\omega\) is a distance (\cref{proposition_optimal_omega}).
In general, the optimal constant in the unweighted Morrey--Sobolev inequality \eqref{eq_izaitheungae7ubahj4eeSah} is not known~\cite{Hynd_Seuffert_2021}.

When \(n = 1\), the estimate turns out to be optimal for compactly supported functions.
\begin{theorem}
\label{theorem_optimal_compact_n_1}
	Let  \(\gamma \in \R\), and \(1 < p < +\infty\).
	There exists a constant \(C > 0\) such that given \(\omega \in C \brk{\R_+ \times \R_+, \intvr{0}{+\infty}}\) for which every \(u \in \smash{C^\infty_c}\brk{\R_+}\) satisfies for almost every \(x\), \(y \in \R_+\), the estimate
	\begin{equation}
	\label{eq_ooshaiyohp6chahshuu9Ahre}
	 \abs{u \brk{x} - u \brk{y}}
	 \le
	 \omega \brk{x, y}\,
	 \brk[\bigg]{\int_{\R_+}
	 \norm{u'\brk{z}}^p {z}^\gamma \dd z}^\frac{1}{p},
	\end{equation}
	then necessarily, for all \(x\), \(y \in \R_+\),
	\[
	   \omega \brk{x, y} \ge C\, \Theta_{1 - \frac{1}{p},1 - \frac{1}{p}, 1 - \frac{1 + \gamma}{p}}
	   \brk{x, y}.
	\]
\end{theorem}

In higher dimensions \(n \ge 2\), the estimate can be improved for functions that are compactly supported in \(\R^n_+\).
We get the following variant of \cref{thm: weighted Morrey--Sobolev}.

\begin{theorem}
\label{thm: weighted Morrey--Sobolev compactly supported}
Let  \(n\geq 2\), \(\gamma \in \R\), and \(n < p < +\infty\).
	There exists a constant \(C = C\brk{n,\gamma,p} > 0\), depending at most on \(n\), \(\gamma\), and \(p\),
	such that, for every \(u \in C^\infty_c \brk{\R^n_+}\)
	and every \(x\), \(y \in \R^n_+\),
	\[
	  \abs{u \brk{x} - u\brk{y}}
 \le C\, \Theta^0_{1-\frac{n}{p}, 1 - \frac{n + \gamma}{p}}\brk{x, y}\,
 \brk[\bigg]{\int_{\R^n_+} \norm{\Diff u\brk{z}}^p {z_n}^\gamma \dd z}^\frac{1}{p}.
	\]
\end{theorem}

The function \(\Theta^0_{\smash{\beta}, \kappa} \colon \R^n_+ \times \R^n_+ \to \intvr{0}{+\infty}\) appearing in \cref{thm: weighted Morrey--Sobolev compactly supported} is
defined for given \(\alpha > 0\), \(\beta \in \R\), and \(\kappa \in \R\), by
\begin{equation}
\label{eq_inoiX5Eiphu5saeQuu6so5ta}
 \Theta^0_{\smash{\beta}, \kappa} \brk{x, y}
 =
 \left\{
 \begin{aligned}
		&\frac{\abs{x - y}^{\beta}}{\min \brk{x_n, y_n}^{-\kappa} \max\brk{x_n, y_n, \abs{x - y}}^{\beta} }
		& &\text{if \(\kappa < 0\),}\\
%	   &
	   & \frac{\max\brk{x_n, y_n}^{\kappa} \abs{x - y}^{\beta}}{\max \brk{x_n, y_n, \abs{x - y} }^{\beta}}  & & \text{if \(\kappa \geq 0\).}
 \end{aligned}
 \right.
\end{equation}

In other words, \cref{thm: weighted Morrey--Sobolev compactly supported} shows that, if \(u\) is compactly supported,
	\begin{itemize}
		\item 
		if \(\gamma > p - n\),
		\begin{equation}
		\label{eq: compact gamma > p - n}
			\norm{u\brk{x} - u \brk{y}}
			\leq C \frac{\abs{x - y}^{1 - \frac{n}{p}}}{\min \brk{x_n, y_n}^{\frac{n + \gamma}{p} - 1} \max\brk{x_n, y_n, \abs{x - y}}^{1 - \frac{n}{p}} }
			 \brk[\bigg]{\int_{\R^n_+}
	 \norm{\Diff u\brk{z}}^p {z_n}^\gamma \dd z}^\frac{1}{p},
		\end{equation}
		\item
		if \(\gamma \le p - n\),
		\begin{equation}
			\norm{u\brk{x} - u \brk{y}}
			\leq C
			\frac{\max\brk{x_n, y_n}^{1 - \frac{n + \gamma}{p}} \abs{x - y}^{1 - \frac{n}{p} }}{\max \brk{x_n, y_n, \abs{x - y} }^{1 - \frac{n}{p}}}
			\brk[\bigg]{\int_{\R^n_+} \norm{\Diff u\brk{z}}^p {z_n}^\gamma \dd z}^\frac{1}{p}.
		\end{equation}
	\end{itemize}

\begin{remark}
Compared with the definition \eqref{eq_hoh1az5ishahXeera3Daepod} of the function \(\Theta_{\alpha, \beta, \kappa}\) appearing in \cref{thm: weighted Morrey--Sobolev}, the only difference in the definition \eqref{eq_inoiX5Eiphu5saeQuu6so5ta} of \(\smash{\Theta^0_{\beta,\kappa}}\) is in the range \(0 \le \kappa \le 1 + \beta - \alpha\).
In fact, the two cases \(\kappa = 0\) and \(0 < \kappa \leq 1 + \beta - \alpha \) merged with the case \(\kappa \ge 1 + \beta - \alpha\), which prevails and yields a stronger inequality. We also note that \(\alpha\) does not play any role in \cref{thm: weighted Morrey--Sobolev compactly supported}.

In terms of weighted Sobolev spaces \(\smash{\dot{W}}^{1,p}_{\smash{\gamma}} \brk{\R^n_+}\), the difference between \cref{thm: weighted Morrey--Sobolev compactly supported,thm: weighted Morrey--Sobolev} arises in the range \(-1 < \gamma \leq p - n\), for which the estimates merge with those of the range \(\gamma \leq - 1\).
The spaces \(\smash{\smash{\dot{W}}^{1,p}_{\smash{\gamma}}} \brk{\R^n_+}\) in the range \(-1 < \gamma \leq p - n\) have traces which are continuous or have vanishing mean oscillations, but are not constant.
On the other hand, even though the theory of traces shows that compactly supported functions are not dense in the range \(p - n < \gamma < p - 1\), the weighted Morrey--Sobolev-type estimates nevertheless have the same form.
\end{remark}

\begin{remark}
In particular when \(\gamma = 0\), \cref{thm: weighted Morrey--Sobolev compactly supported} implies that, if \(u \in C^\infty_c\brk{\R^n_+}\),
\begin{equation}
\abs{u \brk{x} - u \brk{y}}
	 \le
	 C 
	 \min \brk{\max \brk{x_n, y_n}, \norm{x - y}}^{1 - \frac{n}{p}}
	 \brk[\bigg]{\int_{\R^n_+}
	 \norm{\Diff u\brk{z}}^p \dd z}^\frac{1}{p}. \qedhere
	\end{equation}
\end{remark}

\begin{remark}
\label{corollary_strauss}
If \(\gamma < p - n\), it follows from \cref{thm: weighted Morrey--Sobolev compactly supported} that, for each smooth and compactly supported function \(u \in C^\infty_c\brk{\R^n_+}\) and for each \(x \in \R^n_+\),
\[
 \abs{u \brk{x}}
 \le C {x_n}^{1 - \frac{n + \gamma}{p}}
 \brk[\bigg]{\int_{\R^n_+}
	 \norm{\Diff u\brk{z}}^p {z_n}^\gamma \dd z}^\frac{1}{p}
\]
(see \cref{lemma: estimate for compact support}).
\end{remark}

Again, the result of \cref{thm: weighted Morrey--Sobolev compactly supported} is optimal:
\begin{theorem}
\label{theorem_optimal_compact}
	Let  \(n\geq 2\), \(\gamma \in \R\), and \(n < p < +\infty\).
	There exists a constant \(C > 0\) such that, if \(\omega \in C \brk{\R^n_+ \times \R^n_+, \intvr{0}{+\infty}}\) satisfies, for every \(u \in C^\infty_c \brk{\R_+^n}\) and for almost every \(x\), \(y \in \R^n_+\), the estimate
	\begin{equation}
	\label{eq: hypothesis on omega compact}
	 \abs{u \brk{x} - u \brk{y}}
	 \le
	 \omega \brk{x, y}
	 \brk[\bigg]{\int_{\R^n_+}
	 \norm{\Diff u\brk{z}}^p {z_n}^\gamma \dd z}^\frac{1}{p},
	\end{equation}
	then necessarily, for all \(x\), \(y \in \R^n_+\),
	\[
	   \omega \brk{x, y} \ge C \Theta^0_{\smash{1 - \frac{n}{p}, 1 - \frac{n + \gamma}{p}}}
	   \brk{x, y}.
	\]
\end{theorem}

When \(\gamma = p - n\), the weighted Sobolev space \(\smash{\dot{W}}^{1,p}_\gamma\brk{\R^n_+}\) corresponds to the Sobolev space \(\smash{\dot{W}}^{1,p}\brk{\Hset^n}\) on the \(n\)-dimensional hyperbolic space \(\Hset^n\) in the Poincaré half-space model. We obtain, in \cref{section: hyperbolic Morrey--Sobolev}, the following estimate and optimality result.
\begin{theorem}
	\label{theorem_hyperbolic}
	Let \(n\geq 1\) and \(n < p < +\infty\). There exists a constant \(C = C\brk{n,p} > 0\), depending at most on \(n\) and \(p\), such that, for every \(u \in \smash{\dot{W}}^{1, p}\brk{\Hset^n}\) and for almost every \(x\), \(y \in \Hset^n\),
	\[
	\abs{u \brk{x} - u \brk{y}}
	\le C
	\max \brk{\mathrm{d}_{\Hset^n} \brk{x, y}^{1 - \frac{n}{p}}, \mathrm{d}_{\Hset^n} \brk{x, y}^{1 - \frac{1}{p}}}
	\brk[\bigg]{\int_{\Hset^n}\abs{\Diff u}^p}^{\frac{1}{p}}.
	\]
	Moreover, if \(\theta \colon \intvr{0}{+\infty} \to \intvr{0}{+\infty}\) satisfies, for every \(u \in \smash{\dot{W}}^{1, p}\brk{\Hset^n}\) and for almost every \(x\), \(y \in \Hset^n\),
	\[
	\abs{u \brk{x} - u \brk{y}}
	\le \theta \brk{\mathrm{d}_{\Hset^n} \brk{x, y}}
	\brk[\bigg]{\int_{\Hset^n}\abs{\Diff u}^p}^{\frac{1}{p}},
	\]
	then there exists a constant \(C > 0\) such that
	\[
	\theta \brk{t} \geq C \max \brk{t^{1 - \frac{n}{p}}, t^{1 - \frac{1}{p}}}.
	\]
\end{theorem}
When \(n = 1\), the hyperbolic space \(\Hset^1\) is isometric to the real line \(\R\), and we recover the classical one-dimensional Morrey--Sobolev embedding \eqref{eq_izaitheungae7ubahj4eeSah}.

In the spirit of \cref{thm: weighted Morrey--Sobolev compactly supported}, we have a variant of \cref{theorem_hyperbolic} for compactly supported functions on the hyperbolic space \(\Hset^n\).
\begin{theorem}
\label{theorem_hyperbolic_compact}
Let \(n\geq 2\) and \(n < p < +\infty\). There exists a constant \(C = C\brk{n,p} > 0\), depending at most on \(n\) and \(p\), such that, for every \(u \in C^\infty_c \brk{\Hset^n}\) and for every \(x\), \(y \in \Hset^n\),
\[
  \abs{u \brk{x} - u \brk{y}}
  \le C
  \min \brk{\mathrm{d}_{\Hset^n} \brk{x, y}^{1 - \frac{n}{p}},  1}
 \brk[\bigg]{\int_{\Hset^n}\abs{\Diff u}^p}^{\frac{1}{p}}.
\]
Moreover, if \(\theta \colon \intvr{0}{+\infty} \to \intvr{0}{+\infty}\) satisfies, for every \(u \in \smash{\dot{W}}^{1, p}\brk{\Hset^n}\) and for almost every \(x\), \(y \in \Hset^n\),
\[
  \abs{u \brk{x} - u \brk{y}}
  \le \theta \brk{\mathrm{d}_{\Hset^n} \brk{x, y}}
 \brk[\bigg]{\int_{\Hset^n}\abs{\Diff u}^p}^{\frac{1}{p}},
\]
then there exists a constant \(C > 0\) such that
\[
\theta \brk{t} \geq C \min \brk{t^{1 - \frac{n}{p}}, 1}.
\]
\end{theorem}

\section{Weighted Morrey--Sobolev inequality}
\label{section: weighted MS}
We begin with an estimate on averages on a cone based on the Sobolev representation formula.

\begin{lemma}
	\label{lemma: Sobolev representation estimate}
	Let  \(n \geq 1\). If \(n < p < +\infty\), then, for all \(\gamma \in \R\),
	there exists a positive constant \(C = C\brk{n,\gamma, p} > 0\) such that, if \(u \in \smash{\dot{W}}^{1,p}_\gamma\brk{\R^n_+}\), then, for almost all \(x = (x', x_n) \in \R^n_+\) and all \(R > 0\), one has
	\begin{equation}
	\label{eq_dagheesh0Sae7iebethoo9no}
		\fint_{C_{R}\brk{x}} \norm{u\brk{x} - u} 
		\leq 
		C \brk[\bigg]{\int_{C_R\brk{x}}
			\norm{\Diff u\brk{z}}^p {z_n}^\gamma \dd z}^\frac{1}{p} 
		\times
		\begin{cases}
			\min \brk{x_n, R}^{1 - \frac{n}{p}} x_n{}^{-\frac{\gamma}{p}}
			&\text{if } \gamma  > p - n, \\
			\ln \brk[\Big]{1 + \brk[\big]{\frac{R}{x_n}}^\frac{p - n}{p - 1}}^{1 - \frac{1}{p}}
			&\text{if } \gamma = p - n,\\
			\max \brk{x_n, R}^{-\frac{\gamma}{p}}R^{1 - \frac{n}{p}}
			&\text{if } \gamma < p - n.
		\end{cases}  	
	\end{equation}
\end{lemma}

Here and in the sequel, we have defined, for any \(x \in \R^n\) and \(R > 0\), the cone 
\begin{equation*}
	C_R\brk{x} = \set[\big]{z \in \R^n \st \norm{z^\prime - x^\prime} < z_n - x_n < R}.
\end{equation*}
A straightforward computation using spherical coordinates shows that the volume of this cone depends only on \(R\):
\begin{equation}
	\label{eq: volume of cone}
	\begin{split}
		\norm{C_R\brk{x}} = \norm{C_R\brk{0}} 
		& = \int_{\Ball^{n-1}_R} \int_{\norm{z^\prime}}^R \dd z_n \dd z^\prime \\
		& = \frac{\norm{\Ball^{n-1}}}{n} R^n.
	\end{split}
\end{equation}
\begin{proof}[Proof of \cref{lemma: Sobolev representation estimate}]
	By the Sobolev representation formula and H\"older's inequality of exponents \(p\) and \(p^\prime = \frac{p}{p - 1}\), we have 
	\begin{equation}
		\label{eq: applying Sobolev representation}
		\begin{split}
		\fint_{C_{R}\brk{x}} \norm{u\brk{x} - u} 
		& \leq C \int_{C_{R}\brk{x}} \frac{\norm{\Diff u\brk{z}}}{\norm{x - z}^{n - 1}} \dd z \\
		& \leq C\brk[\bigg]{\int_{\R^n_+} \norm{\Diff u\brk{z}}^p {z_n}^\gamma \dd z}^\frac{1}{p} 
		\brk[\bigg]{\int_{C_R\brk{x}} \frac{{z_n}^{- \frac{\gamma}{p - 1}}}{\norm{x - z}^{\brk{n - 1}p^\prime}} \dd z}^{\frac{1}{p^\prime}}.
		\end{split}
	\end{equation}
	To estimate the second integral on the right-hand side of \eqref{eq: applying Sobolev representation}, we may assume, without loss of generality, that \(x = \brk{0, x_n} \in \R^{n-1} \times \brk{0, +\infty}\).
	By Tonelli's theorem and a change of variable to polar coordinates, we have 
	\begin{equation}
		\label{eq: tonelli}
		\begin{split}
		\int_{C_R\brk{x}} \frac{{z_n}^{- \frac{\gamma}{p-1}}}{\norm{x-z}^{\brk{n-1}p^\prime}} \dd z 
		&= \int_{x_n}^{x_n + R} \int_{\Ball^{n-1}_{z_n - x_n}} \frac{{z_n}^{- \frac{\gamma}{p-1}}}{\brk{\abs{z^\prime}^2 + \brk{z_n - x_n}^2}^{\brk{n - 1}p^\prime/2}} \dd z^\prime \dd z_n\\
		& \leq \brk{n-1}\norm{\Ball^{n-1}} \int_{x_n}^{x_n + R} \int_0^{z_n - x_n} \frac{{z_n}^{- \frac{\gamma}{p-1}}r^{n - 2}}{\brk{z_n - x_n}^{\brk{n - 1}p^\prime}}\dd r \dd z_n. 
		\end{split}
	\end{equation}
	A direct computation yields
	\begin{equation*}
		\int_0^{z_n - x_n} \frac{r^{n - 2}}{\brk{z_n - x_n}^{\brk{n - 1}p^\prime}}\dd r
		=
		\frac{1}{\brk{n - 1} \brk{z_n - x_n}^{\frac{n - 1}{p - 1}}},
	\end{equation*}
	which we inject into \eqref{eq: tonelli} to reach
	\begin{equation}
		\label{eq: estimate with max}
		\begin{split}
			\int_{C_R\brk{x}} \frac{{z_n}^{- \frac{\gamma}{p-1}}}{\abs{x - z}^{\brk{n - 1}p^\prime}} \dd  z
			& \leq \norm{\Ball^{n-1}}
			 \int_{x_n}^{x_n + R} \frac{{z_n}^{- \frac{\gamma}{p-1}}}{\brk{z_n - x_n}^\frac{n - 1}{p- 1}} \dd z_n \\
			& = \norm{\Ball^{n - 1}} \int_0^R \frac{\brk{t + x_n}^{- \frac{\gamma}{p-1}}} {t^\frac{n - 1}{p- 1}} \dd t \\
			& \leq \max\brk{1, 2^{- \frac{\gamma}{p-1}}} \norm{\Ball^{n-1}} \int_0^R \frac{\brk{\max\brk{t, x_n}}^{- \frac{\gamma}{p-1}}} {t^\frac{n - 1}{p- 1}} \dd t,
		\end{split}
	\end{equation}
	since, for every \(t \ge 0\),
\[
 \max\brk{t,x_n} \leq t + x_n \leq 2 \max\brk{t ,x}.
\]
If \(x_n \ge R\), we compute
	\begin{equation}
		\label{eq: xn > R}
		\begin{split}
		\int_0^R \frac{\brk{\max\brk{t, x_n}}^{- \frac{\gamma}{p-1}}} {t^\frac{n - 1}{p- 1}} \dd t
		&= \int_0^R \frac{x_n{}^{- \frac{\gamma}{p-1}}}{t^\frac{n - 1}{p - 1}} \dd t\\
		&= \frac{p-1}{p-n} \brk[\Big]{x_n{}^{-\frac{\gamma}{p}} R^{1-\frac{n}{p}}}^{p^\prime}\\
		&\leq 
		\left\{
			\begin{aligned}
				&\frac{p-1}{p - n} \brk[\Big]{\min \brk{x_n, R}^{1 - \frac{n}{p}}  x_n{}^{-\frac{\gamma}{p}}}^{p^\prime}
				&\text{if } \gamma > p - n, \\
				&\frac{p-1}{\brk{p-n}\ln 2} \ln \brk[\Big]{1 + \brk[\Big]{\frac{R}{x_n}}^{\frac{p - n}{p - 1}}}
				&\text{if }\gamma = p - n,\\
				&\frac{p-1}{p - n} \brk[\Big]{\max \brk{x_n, R}^{-\frac{\gamma}{p}}
				R^{1 - \frac{n}{p}}}^{p^\prime}
				&\text{if }\gamma < p - n,
			\end{aligned}\right.		
	    \end{split}
	\end{equation}
using when \(\gamma = p - n\) the fact that, by concavity of the logarithm, 
\[
\brk[\Big]{\frac{R}{x_n}}^{\frac{p - n}{p - 1}} \ln 2
\leq 
 \ln \brk[\Big]{1 + \brk[\Big]{\frac{R}{x_n}}^{\frac{p - n}{p - 1}}}.
\]
If \(x_n \le R\), we have
	\begin{equation*}
		\begin{split}
		\int_0^R \frac{\brk{\max\brk{t, x_n}}^{- \frac{\gamma}{p-1}}} {t^\frac{n - 1}{p- 1}} \dd t
			& = \int_0^{x_n}  \frac{x_n{}^{- \frac{\gamma}{p-1}}} {t^\frac{n - 1}{p- 1}}\dd t
			+ \int_{x_n}^R  \frac{t^{- \frac{\gamma}{p-1}}} {t^\frac{n - 1}{p- 1}} \dd t \\
			& =  
			\frac{p-1}{p-n} x_n{}^{\brk{1 - \frac{n + \gamma}{p}}p^\prime } 
			+\left\{\begin{aligned}
				&\frac{p-1}{\gamma + n -p} x_n{}^{\brk{1 - \frac{n+\gamma}{p}}p^\prime } 
				&\text{if } \gamma > p - n,\\
				&\ln \frac{R}{x_n} 
				&\text{if }\gamma = p - n,\\
				&\frac{p-1}{p - n -\gamma }R^{\brk{1 - \frac{n+\gamma}{p}} p^\prime} 
				&\text{if }\gamma < p - n
			\end{aligned}
			\right.
	\end{split}
	\end{equation*}
	and then 
	\begin{equation}
		\label{eq: xn < R}
		\int_0^R \frac{\brk{\max\brk{t, x_n}}^{- \frac{\gamma}{p-1}}} {t^\frac{n - 1}{p- 1}} \dd t \leq
			\left\{
			\begin{aligned}
				&\frac{\brk{p-1}\gamma}{\brk{p-n}\brk{\gamma + n - p}} \brk[\Big]{\min \brk{x_n, R}^{1 - \frac{n}{p}}  x_n{}^{-\frac{\gamma}{p}}}^{p^\prime}
				&\text{if } \gamma > p - n, \\
				&\frac{p - 1}{p-n}\brk[\Big]{\frac{1}{\ln 2} + 1} \ln\brk[\Big]{1 + \brk[\Big]{\frac{R}{x_n}}^{\frac{p - n}{p - 1}}}
				&\text{if }\gamma = p - n,\\
				&\frac{\brk{p-1}\brk{2p - 2n - \gamma}}{\brk{p - n}\brk{p - n - \gamma}} \brk[\Big]{\max \brk{x_n, R}^{-\frac{\gamma}{p}}
				R^{1 - \frac{n}{p}}}^{p^\prime}
				&\text{if }\gamma < p - n,
			\end{aligned}
			\right.
	\end{equation}
	since 
	\[
	  1 + \ln \brk[\Big]{\brk[\Big]{\frac{R}{x_n}}^{\frac{p - n}{p - 1}}}
	  \le \brk[\Big]{\frac{1}{\ln 2}  + 1} \ln \brk[\Big]{1 + \brk[\Big]{\frac{R}{x_n}}^{\frac{p - n}{p - 1}}}.
	\]
Combining the estimates \eqref{eq: estimate with max}, \eqref{eq: xn > R}, and \eqref{eq: xn < R}, and injecting them into  \eqref{eq: applying Sobolev representation}, we obtain the conclusion \eqref{eq_dagheesh0Sae7iebethoo9no}.
	\end{proof}

\begin{proof}[Proof of \cref{thm: weighted Morrey--Sobolev} when \(\gamma > -1\) (inequalities \eqref{thm: weighted Morrey--Sobolev gamma > p - n}, \eqref{thm: weighted Morrey--Sobolev gamma = p - n}, and \eqref{thm: weighted Morrey--Sobolev gamma < p - n})]
	Let \(x\), \(y \in \R^{n}_+\) and fix \(R = 2 \norm{x-y}\).   
	Then \(\norm{C_{R}\brk{x} \cap C_{R}\brk{y}} > 0 \) and
	\begin{equation}
		\label{eq: initial estimate for u(x)-u(y)}
		\begin{split}
			\norm{u\brk{x} - u\brk{y}} 
			& \leq \fint_{C_{R}\brk{x}\cap C_{R}\brk{y}} \norm{u\brk{x} - u}  + \fint_{C_{R}\brk{x}\cap C_{R}\brk{y}} \norm{u\brk{y} - u} \\
			& \leq \frac{\norm{C_R\brk{0}}}{\norm{C_{R}\brk{x}\cap C_R\brk{y}}} \brk[\Big]{ \fint_{C_{R}\brk{x}} \norm{u\brk{x} - u} + \fint_{C_{R}\brk{y}} \norm{u\brk{y} - u}}.
		\end{split}
	\end{equation}
	Setting \(z = \frac{x + y}{2} + \brk{0, \abs{x - y}}\), we have \(C_{\frac{R}{4}}\brk{z} \subset C_R\brk{x} \cap C_R\brk{y}\) and thus, in view of \eqref{eq: volume of cone},
	\begin{equation*}
		\frac{\norm{C_R\brk{0}}}{\norm{C_R\brk{x} \cap C_R\brk{y}}} \leq \frac{\norm{C_R\brk{0}}}{\norm{C_{R/4}\brk{0}}} = 4^n.
	\end{equation*}
	Using \cref{lemma: Sobolev representation estimate} in \eqref{eq: initial estimate for u(x)-u(y)}, we have 
	\begin{equation}
		\label{eq: applying Sobolev representation formula}
		\begin{split}
			\norm{u\brk{x} - u\brk{y}} 
			&\leq 4^n C  
			\brk[\bigg]{\int_{\R^n_+}
				\norm{\Diff u\brk{z}}^p {z_n}^\gamma \dd z}^\frac{1}{p} \\
			&\quad \qquad  \times
			\begin{cases}
				\min \brk{x_n, R}^{1 - \frac{n}{p}} x_n{}^{-\frac{\gamma}{p}} + \min \brk{y_n, R}^{1 - \frac{n}{p}} {y_n}^{-\frac{\gamma}{p}}
				&\text{if } \gamma  > p - n, \\
				\ln \brk[\Big]{1 + \brk[\big]{\frac{R}{x_n}}^\frac{p - n}{p - 1}}^{\frac{1}{p^\prime}} + \ln \brk[\Big]{1 + \brk[\big]{\frac{R}{y_n}}^\frac{p - n}{p - 1}}^{\frac{1}{p^\prime}}
				&\text{if } \gamma = p - n,\\
				\max \brk{x_n, R}^{-\frac{\gamma}{p}}R^{1 - \frac{n}{p}} + \max \brk{y_n, R}^{-\frac{\gamma}{p}}R^{1 - \frac{n}{p}}
				&\text{if } \gamma < p - n.
			\end{cases}
		\end{split}
	\end{equation} 
	Using, in \eqref{eq: applying Sobolev representation formula}, the fact that
	\begin{gather*}
	 \max \brk{x_n, \abs{x - y}} \le \max \brk{x_n, 2 \abs{x - y}} \le 2\max \brk{x_n, \abs{x - y}},\\
	 \min \brk{x_n, \abs{x - y}} \le \min \brk{x_n, 2 \abs{x - y}} \le 2\min \brk{x_n, \abs{x - y}},\\
	 \max\brk{x_n, \abs{x - y}} \le \max\brk{x_n, y_n, \abs{x - y}} \le 2 \max \brk{x_n, \abs{x - y}},
	 \intertext{and}
	 \max\brk{y_n, \abs{x - y}} \le \max\brk{x_n, y_n, \abs{x - y}} \le 2 \max \brk{y_n, \abs{x - y}},
	\end{gather*}
	and that, for \eqref{thm: weighted Morrey--Sobolev gamma > p - n},
	\[
	\begin{split}
	 \max \brk{\brk{x_n, \abs{x - y}}^{1 - \frac{n}{p}} x_n{}^{-\frac{\gamma}{p}}&, \min \brk{y_n, \abs{x - y}}^{1 - \frac{n}{p}} {y_n}^{-\frac{\gamma}{p}}}\\
	 & = \abs{x - y}^{1 - \frac{n}{p}} \max \brk[\Big]{\frac{{x_n}^{1 - \smash{\frac{n+\gamma}{p}}}}{\max \brk{x_n, \abs{x - y}}^{1 - \smash{\frac{n}{p}}}}, \frac{{x_n}^{1 - \smash{\frac{n+\gamma}{p}}}}{\max \brk{x_n, \abs{x - y}}^{1 - \frac{n}{p}}}}\\
	 &\le \frac{2^{1 - \frac{n}{p}}\abs{x - y}^{1 - \frac{n}{p}}}{\max \brk{x_n, y_n, \abs{x - y}}^{1 - \frac{n}{p}}} \max \brk{{x_n}^{1 - \smash{\frac{n+\gamma}{p}}}, {y_n}^{1 - \smash{\frac{n+\gamma}{p}}}}\\
	 &=\frac{2^{1 - \frac{n}{p}}\abs{x - y}^{1 - \frac{n}{p}}}{\max \brk{x_n, y_n, \abs{x - y}}^{1 - \frac{n}{p}} \min \brk{x_n, y_n}^{\frac{n+\gamma}{p} - 1}},
	 \end{split}
	\]
    yields the desired conclusions \eqref{thm: weighted Morrey--Sobolev gamma > p - n}, \eqref{thm: weighted Morrey--Sobolev gamma = p - n}, and \eqref{thm: weighted Morrey--Sobolev gamma < p - n} and ends the proof.
\end{proof}

We note that we have not used the assumption \(\gamma > -1\) in the above proof.
It turns out that we will improve the estimate \eqref{thm: weighted Morrey--Sobolev gamma < p - n} when \(\gamma \le -1\) thanks to the observation that Sobolev functions then have constant traces.
\begin{proposition}
	\label{prop: singular traces}
	Let \(n \geq 1\) and \(1 \leq p < +\infty\).
	If  \(\gamma < p - 1\), then, for any \(u \in \smash{\dot{W}}^{1,p}_{\smash{\gamma}} \brk{\R^n_+}\), we have \(u \in \smash{ \smash{\dot{W}}^{1,1}_{\smash{\mathrm{loc}}}\brk{\smash{\overline{\R}}^n_+}}\) and
	its trace \(v = \operatorname{tr}_{\R^{n - 1}} u\) satisfies
	\begin{equation}
		\label{eq_kah1tuopah2ohpui0eD8aefe}
		\smashoperator[r]{\iint_{\R^{n - 1}\times \R^{n - 1}}} \frac{\abs{v \brk{x} - v \brk{y}}^p}{\abs{x - y}^{\brk{n - 1} + p  - \gamma}} \dd x \dd y
		\le C
		\int_{\R^n_+} \abs{\Diff u \brk{z}}^p z_n{}^\gamma \dd z.
	\end{equation}
	In particular, if \(\gamma \le - 1\), then
	\(v\) is constant.
\end{proposition}

Since our aim is to treat weakly differentiable functions up to the boundary, a density argument would require to approximate Sobolev functions with a quite singular weight --- it is not locally integrable near the origin. We rather perform a direct argument from trace theory.

\begin{proof}[Proof of \cref{prop: singular traces}]
	By H\"older’s inequality, we have, if the set \(K \subseteq \smash{\overline{\R}}^n_+\) is compact,
	\[
	\int_{K}
	\abs{\Diff u }
	\le
	\brk[\bigg]{\int_{\R^n_+} \abs{\Diff u \brk{z}}^p z_n{}^\gamma \dd z}^\frac{1}{p}
	\brk[\bigg]{\int_{K} \frac{1}{z_n{}^\frac{\gamma}{p - 1}} \dd z}^{1 - \frac{1}{p}} < +\infty,
	\]
	provided \(\gamma < p - 1\) so that \(u \in W^{1, 1}_{\smash{\mathrm{loc}}} \brk{\smash{\overline{\R}}^{n}_+}\).
	
	The estimate \eqref{eq_kah1tuopah2ohpui0eD8aefe} follows from the classical proof of the trace estimates in weighted Sobolev spaces~\cite{Mironescu_Russ_2015}*{Thm.\ 1.3 \& Rem.\ 3.2}.
	
	The last conclusion follows from a direct argument~\cite{Mironescu_Russ_2015}*{Prop. 5.1} or the Bourgain--Brezis--Mironescu characterisation of constant functions~\citelist{\cite{Brezis_2002}\cite{Bourgain_Brezis_Mironescu_2001}\cite{DeMarco_Mariconda_Solimini_2008}\cite{RanjbarMotlagh_2020}}.
\end{proof}

\begin{proof}[Proof of \cref{thm: weighted Morrey--Sobolev} when \(\gamma \le - 1\) (inequality \eqref{thm: weighted Morrey--Sobolev gamma <= -1})]
	Since \(\gamma \leq -1 < p - n\) and since the proof of \cref{thm: weighted Morrey--Sobolev} \eqref{thm: weighted Morrey--Sobolev gamma < p - n} does not use the assumption \(\gamma > - 1\),
	we have, for almost every \(x\), \(y \in \R^n_+\),
	\begin{equation}
		\label{eq: gamma < p - n}
		\norm{u\brk{x} - u\brk{y}} \leq C \abs{x - y}^{1 - \frac{n}{p}} \max \brk{x_n, y_n, \abs{x - y}}^{-\frac{\gamma}{p}} \brk[\bigg]{\int_{\R^n_+}
	 \norm{\Diff u\brk{z}}^p {z_n}^\gamma \dd z}^\frac{1}{p}.
	\end{equation}
	If \(\max \brk{x_n, y_n, \abs{x -y}} = \max \brk{x_n, y_n}\), the inequality \eqref{eq: gamma < p - n} is equivalent to 
    \begin{equation}
		\label{eq_EsooQua7eitiv2sahyas1Sha}
		\norm{u\brk{x} - u\brk{y}} \leq C 
    \frac{\max\brk{x_n, y_n}^{1 - \frac{n + \gamma}{p}} \abs{x - y}^{1 - \frac{n}{p} }}{\max \brk{x_n, y_n, \abs{x - y} }^{1 - \frac{n}{p}}} \brk[\bigg]{\int_{\R^n_+}
	 \norm{\Diff u\brk{z}}^p {z_n}^\gamma \dd z}^\frac{1}{p}.
	\end{equation}
	
		Assuming that \(\operatorname{tr}_{\R^{n - 1}} u = u \vert_{\R^{n - 1}}\),
	this estimate also holds for almost every \(x \in \R^{n - 1}\) and for almost every \(y \in \R^n_+\).
	By \cref{prop: singular traces}, since \(\gamma \leq - 1\), for almost every \(x\), \(y \in \R^n_+\), we have \(u\brk{x^\prime,0} = u\brk{y^\prime, 0}\). 
	Therefore, by the triangle inequality and by \eqref{eq: gamma < p - n}, we have
	\begin{equation}
		\label{eq: triangle inequality}
		\begin{split}
		\norm{u\brk{x} - u\brk{y}} 
		& \leq \norm{u\brk{x} - u\brk{x^\prime, 0}} + \norm{u\brk{y} - u\brk{y^\prime, 0}}\\
		& \leq C \brk[\big]{{x_n}^{1 - \smash{\frac{n+\gamma}{p}}} + {y_n}^{1- \smash{\frac{n+\gamma}{p}}}}\brk[\bigg]{\int_{\R^n_+}
	 \norm{\Diff u\brk{z}}^p {z_n}^\gamma \dd z}^\frac{1}{p}\\
		& \leq 2C \max\brk{x_n, y_n}^{1 - \frac{n+\gamma}{p}}\brk[\bigg]{\int_{\R^n_+}
	 \norm{\Diff u\brk{z}}^p {z_n}^\gamma \dd z}^\frac{1}{p}\\
	 &= 2C \frac{\max\brk{x_n, y_n}^{1 - \frac{n + \gamma}{p}} \abs{x - y}^{1 - \frac{n}{p} }}{\max \brk{x_n, y_n, \abs{x - y} }^{1 - \frac{n}{p}}} \brk[\bigg]{\int_{\R^n_+} \norm{\Diff u\brk{z}}^p {z_n}^\gamma \dd z}^\frac{1}{p},
		\end{split}
	\end{equation}
	provided \(\max\brk{x_n, y_n, \abs{x - y}} = \abs{x - y}\).
	The conclusion follows by \eqref{eq_EsooQua7eitiv2sahyas1Sha} and \eqref{eq: triangle inequality} when \(\abs{x - y} \le \max \brk{x_n, y_n}\) and \(\abs{x - y}\ge \max \brk{x_n, y_n}\) respectively.
\end{proof}

In order to prove \cref{thm: weighted Morrey--Sobolev compactly supported}, we first prove an estimate involving the distance to the boundary.

\begin{lemma}
	\label{lemma: estimate for compact support}
	Let \(n \geq 1\), \(n < p < +\infty\) and \(\gamma < p - 1\).
	There exists a constant \(C = C\brk{n, \gamma , p } > 0\) depending at most on \(n\), \(\gamma\), and \(p\), such that, for every \(u \in C^\infty_c\brk{\R^n_+}\) and every \(x \in \R^n_+\),
	\begin{equation}
	\label{eq_leephie0joilutohFoPhu1Pe}
		\norm{u\brk{x}}
		\leq C\,x_n^{1 - \frac{n + \gamma}{p}} \brk[\bigg]{\int_{\R^n_+} \norm{\Diff u\brk{z}}^p {z_n}^\gamma \dd z}^{\frac{1}{p}}.
	\end{equation}
\end{lemma}
Even though the estimate keeps the same form for \(\gamma < p - 1\), it gives a continuity property near the boundary when \(\gamma < p- n\), gives a uniform bound when \(\gamma = p - n\), and controls the explosion rate near the boundary when \(\gamma > p - n\).

\begin{proof}[Proof of \cref{lemma: estimate for compact support}]
	Without loss of generality, we may assume \(x = \brk{0, x_n} \in \R^n_+\).
	Then, since \(u\) is compactly supported in \(\R^n_+\), we have, for any \(h \in \Ball^{n - 1}_{x_n}\),
	\begin{equation}
		\label{eq: applying FTC}
			\norm{u\brk{0,x_n}}
			\leq
			\int_0^1 \norm{\Diff u \brk{ \brk{1 - t} h, t x_n}} \norm{\brk{h, x_n}} \dd t.
	\end{equation}
	Taking the average over all \(h \in \Ball^{n - 1}_{x_n}\) in \eqref{eq: applying FTC} and applying the change of variable \(z = \brk{z^\prime, z_n} = \brk{\brk{1 - t}h, t x_n}\), we obtain
	\begin{equation}
		\label{eq: taking the mean over h}
		\begin{split}
			\norm{u\brk{0,x_n}}
			& \leq \fint_{\Ball^{n - 1}_{x_n}} \int_0^1 \norm{\Diff u \brk{ \brk{1 - t} h, t x_n}} \norm{\brk{h, x_n}} \dd t \dd h\\
			& = \frac{1}{\norm{\Ball^{n - 1}}} \int_{U\brk{x}} \frac{\norm{\Diff u\brk{z}}}{\brk{x_n - z_n}^{n - 1}} \norm{\brk{ \tfrac{z^\prime}{x_n - z_n} , 1}} \dd z \\
			& \leq \frac{\sqrt{2}}{\norm{\Ball^{n - 1}}} \int_{U\brk{x}} \frac{\norm{\Diff u\brk{z}}}{\brk{x_n - z_n}^{n - 1}} \dd z,
		\end{split}
	\end{equation}
where we have defined, for any \(x \in \R^n_+\), the cone
\begin{equation*}
	U\brk{x} = \set[\big]{z \in \R^n \st \norm{z^\prime - x^\prime} < x_ n - z_n < x_n}.
\end{equation*}
Applying H\"older's inequality of exponents \(p\) and \(p^\prime\) to \eqref{eq: taking the mean over h}, we have
\begin{equation}
	\label{eq: applying Holder to upright cone}
	\norm{u\brk{0,x_n}}
	\leq
	\frac{\sqrt{2}}{\norm{\Ball^{n - 1}}} \brk[\bigg]{\int_{\R^{n}_+} \norm{\Diff u \brk{z}}^p {z_n}^\gamma \dd z}^{\frac{1}{p}}
	\brk[\bigg]{\int_{U\brk{x}} \frac{{z_n}^{- \frac{\gamma}{p - 1}}}{\brk{x_n - z_n}^{\brk{n - 1}p^\prime}} \dd z}^{\frac{1}{p^\prime}}.
\end{equation}
Using Tonelli's theorem, we obtain, since \(n < p\) and \(\gamma < p - 1\),
\begin{equation}
\label{eq_aeNgaeg2elohC1ceiJej8mee}
	\begin{split}
	\int_{U\brk{x}} \frac{{z_n}^{- \frac{\gamma}{p - 1}}}{\brk{x_n - z_n}^{\brk{n - 1}p^\prime}} \dd z
	&=
	\int_0^{x_n}\int_{\Ball^{n-1}_{x_n - z_n}} \frac{{z_n}^{- \frac{\gamma}{p - 1}}}{\brk{x_n - z_n}^{\brk{n - 1}p^\prime}} \dd z^\prime \dd z_n \\
	& = \norm{\Ball^{n - 1}} \int_0^{x_n} \frac{{z_n}^{- \frac{\gamma}{p - 1}}}{\brk{x_n - z_n}^{\frac{n - 1}{p - 1}}} \dd z_n\\
	&\le \norm{\Ball^{n - 1}}
		\brk[\bigg]{\int_0^{x_n/2} \frac{z_n^{-\frac{\gamma}{p-1}}}{\brk{x_n/2}^{\frac{n - 1}{p - 1}}}  \dd z_n
		+
		\int_{x_n/2}^{x_n}
		\frac{
			\max\brk{1, 2^\frac{\gamma}{p - 1}}x_n
		}
		{
			\brk{x_n - y_n}^{\frac{n - 1}{p - 1}}
		} \dd z_n} \\
		&  = 2^{\brk{1 - \frac{n + \gamma}{p}}p^\prime} \norm{\Ball^{n - 1}}
	\brk[\Big]{
		\frac{p - 1}{p-1 - \gamma }
		+
		\frac{p - 1}{p - n}\max\brk{1, 2^\frac{\gamma}{p - 1}}}
	x_n^{\brk{1 - \frac{n + \gamma}{p}} p^\prime}.
	\end{split}
\end{equation}
The conclusion \eqref{eq_leephie0joilutohFoPhu1Pe} then follows from \eqref{eq: applying Holder to upright cone} and \eqref{eq_aeNgaeg2elohC1ceiJej8mee}.
\end{proof}

\begin{proof}[Proof of \cref{thm: weighted Morrey--Sobolev compactly supported} when \(\gamma \le p - n\)]
We first note that if
\begin{equation}
\label{eq_ree8ge2eic1see4aiZah5ewe}
 \norm{x - y} \le \frac{1}{2}\max \brk{x_n, y_n}
\end{equation}
we have
\[
  \max \brk{x_n, y_n, \norm{x - y}}
  = \max \brk{x_n, y_n}
\]
and
\[
 \tfrac{1}{2}  \min \brk{x_n, y_n} \leq \max \brk{x_n, y_n} - \norm{x - y}
 \leq \max \brk{x_n, y_n},
\]
and it follows from \cref{thm: weighted Morrey--Sobolev} that
for every \(x, y \in \R^n_+\) satisfying \eqref{eq_ree8ge2eic1see4aiZah5ewe}, we have by monotonicity properties of powers and concavity of the logarithm,
\begin{equation}
	\label{eq: monotonicity+concavity}
 			\norm{u\brk{x} - u \brk{y}}
			\leq C
			\frac{\max\brk{x_n, y_n}^{1 - \frac{n + \gamma}{p}} \abs{x - y}^{1 - \frac{n}{p} }}{\max \brk{x_n, y_n, \abs{x - y} }^{1 - \frac{n}{p}}}
			\brk[\bigg]{\int_{\R^n_+} \norm{\Diff u\brk{z}}^p {z_n}^\gamma \dd z}^\frac{1}{p}.
\end{equation}

On the other hand, if \eqref{eq_ree8ge2eic1see4aiZah5ewe} fails, we have \(\norm{x - y} > \frac{1}{2}\max\brk{x_n, y_n}\), and thus \(\norm{x - y} \geq \frac{1}{2} \max\brk{x_n, y_n, \norm{x - y}}\). Since \(\gamma \leq p - n < p - 1\), we then have by the triangle inequality and \cref{lemma: estimate for compact support} that
\begin{equation}
	\label{eq: applying estimate for compactly supported functions}
	\begin{split}
		\norm{u\brk{x} - u\brk{y}}
		&\leq \norm{u\brk{x}} + \norm{u \brk{y}}\\
		& \leq C \max\brk{\smash{x_n}^{1 - \frac{n + \gamma}{p}}, \smash{y_n}^{1 - \frac{n + \gamma}{p}}}
		\brk[\bigg]{\int_{\R^n_+} \norm{\Diff u\brk{z}}^p {z_n}^\gamma \dd z}^\frac{1}{p}\\
		& \leq 2^{1- \frac{n}{p}} C
		\frac{\max\brk{x_n, y_n}^{1 - \frac{n + \gamma}{p}} \abs{x - y}^{1 - \frac{n}{p} }}{\max \brk{x_n, y_n, \abs{x - y} }^{1 - \frac{n}{p}}} \brk[\bigg]{\int_{\R^n_+}\norm{\Diff u\brk{z}}^p {z_n}^\gamma \dd z}^\frac{1}{p},
	\end{split}
\end{equation}
which yields the conclusion when \(\norm{x - y} \geq \frac{1}{2} \max\brk{x_n, y_n}\). Combining \eqref{eq: monotonicity+concavity} and \eqref{eq: applying estimate for compactly supported functions} concludes the proof.
\end{proof}

\section{Optimality}
We discuss in this section the optimality of the estimates obtained in the previous section.
\begin{proof}[Proof of \cref{theorem_optimal_compact} when \(\gamma \ne p - n\)]
Let us fix a function \(\psi \in C^\infty_c \brk{\R^n}\) such that
\(\psi = 1\) on \(\Ball^n_{1/2}\) and \(\operatorname{supp} \psi \subseteq \Ball^n_1\).
Given \(R \in \intvo{0}{+\infty}\) and \(y \in \smash{\overline{\R}}^n_+\), we define \(\psi_{y, R} \in C^\infty \brk{\R^n}\) by setting, for \(z \in \R^n_+\),
\begin{equation}
\label{eq_Paem4eezeisochaeRu3tei9a}
  \psi_{y, R} \brk{z} = \psi \brk[\Big]{\frac{z - y}{R}}.
\end{equation}
We compute
\[
 \int_{\R^n_+} \abs{\Diff \psi_{y, R} \brk{z}}^p {z_n}^\gamma \dd  z
 = \frac{1}{R^p}\int_{\Ball^n_R\brk{y} \cap \R^n_+} \norm[\Big]{\Diff \psi \brk[\Big]{\frac{z - y}{R}}}^p {z_n}^{\gamma} \dd z
 \le \frac{C}{R^p} \int_{\Ball^n_R\brk{y}} {z_n}^\gamma \dd z.
\]
If \(R \le \frac{y_n}{2}\), then, for all \(z \in \Ball^n_R\brk{y}\), one has \(\frac{1}{2} y_n \leq z_n \leq \frac{3}{2} y_n\), and therefore
\begin{equation}
\label{eq_quaino8pheiRieWoo0gu4xah}
 \int_{\R^n_+} \abs{\Diff \psi_{y, R} \brk{z}}^p {z_n}^\gamma \dd  z
 \le C R^{n - p} {y_n}^\gamma.
\end{equation}
Taking \(R = \frac{\min\brk{y_n, \abs{x - y}}}{2}\) ensures that the support of \(\psi_R\) is a compact subset of \(\R^n_+\), and the assumption \eqref{eq: hypothesis on omega compact} on \(\omega\) together with \eqref{eq_quaino8pheiRieWoo0gu4xah} then implies that
\begin{equation}
	\label{eq: optimality gamma <= -1}
\begin{split}
  1 = \abs{\psi_{y, R} \brk{x} - \psi_{y, R}\brk{y}}
  &\le
  \omega \brk{x, y}
  \brk[\bigg]{ \int_{\R^n_+} \abs{\Diff \psi_{y, R} \brk{z}}^p {z_n}^\gamma \dd  z
  }^\frac{1}{p}\\
  &\le C
  \frac{\omega \brk{x, y}}{{y_n}^{-\smash{\frac{\gamma}{p}}}\min \brk{y_n, \abs{x - y}}^{1 - \frac{n}{p}}}.
  \end{split}
\end{equation}
Inverting the roles of \(x\) and \(y\) in \eqref{eq: optimality gamma <= -1}, we obtain, since \(1 - \frac{n}{p} > 0\),
\begin{equation}
\label{eq_ci7iaNga9pi0fiekooyoolee}
	\begin{split}
	C\omega\brk{x , y} 
	& \geq 
	\max\brk[\big]{
		{x_n}^{-\smash{\frac{\gamma}{p}}} \min\brk{x_n, \norm{x - y}}^{1-\frac{n}{p}},
		{y_n}^{-\smash{\frac{\gamma}{p}}} \min\brk{y_n, \norm{x - y}}^{1-\frac{n}{p}} 
		} \\
	& = \max\brk[\Big]{
		\frac{{x_n}^{1 -\smash{\frac{n + \gamma}{p}}} \norm{x - y}^{1-\frac{n}{p}}}{\max \brk{x_n, \abs{x - y}}^{1 - \smash{\frac{n}{p}}}},
		\frac{{y_n}^{1 -\smash{\frac{n + \gamma}{p}}} \norm{x - y}^{1-\frac{n}{p}}}{\max \brk{y_n, \abs{x - y}}^{1 - \smash{\frac{n}{p}}}}
		} \\
	& \ge \frac{\max\brk{x_n{}^{1 - \smash{\frac{n + \gamma}{p}}}, y_n{}^{1 - \smash{\frac{n + \gamma}{p}}} } \abs{x - y}^{1 - \smash{\frac{n}{p}}}}{\max \brk{x_n, y_n, \abs{x - y} }^{1 - \frac{n}{p}}}.
	\end{split}
\end{equation}

Noting that
\[
 \max\brk{x_n{}^{1 - \smash{\frac{n + \gamma}{p}}}, y_n{}^{1 - \smash{\frac{n + \gamma}{p}}}}
 =
 \begin{cases}
\min \brk{x_n, y_n}^{1 - \smash{\frac{n + \gamma}{p}}}
 &\text{if \(1 - \frac{n + \gamma}{p} < 0\)},\\
1
 &\text{if \(1 - \frac{n + \gamma}{p} = 0\)},\\
\max \brk{x_n, y_n}^{1 - \smash{\frac{n + \gamma}{p}}}
 &\text{if \(1 - \frac{n + \gamma}{p} >0\)}.
 \end{cases}
\]
we have, by definition \eqref{eq_inoiX5Eiphu5saeQuu6so5ta} of \(\Theta^0_{\beta, \kappa}\),
\begin{equation}
\label{eq_sohl2Thoog3taiciesh2osai}
	C\omega\brk{x , y} \ge  \Theta^0_{\smash{1 - \frac{n}{p}, 1 - \frac{n + \gamma}{p}}}
	   \brk{x, y},
\end{equation}
completing the proof.
\end{proof}

\begin{proof}[Proof of \cref{theorem_optimal_compact_n_1}]
Although the statement \cref{theorem_optimal_compact} requires \(n \ge 2\), one can check that this assumption is not used in the proof. This covers the case \(\gamma \ne p - 1\), and we assume that \(\gamma = p - 1\).

We fix a function \(\psi \in C^\infty_c\brk{\R}\) such that \(\psi \brk{0} = 1\) and \(\psi = 0\) on \(\R \setminus \intvo{-1}{1}\), and we define for \(x\), \(y \in \intvo{0}{+\infty}\) the function
\[
 \psi_{x, y} \brk{z} =
 \psi \brk[\Big]{\frac{\ln \brk{z/y}}{\ln\brk{x/y}}},
\]
so that \(\psi_{x, y} \brk{x} = 0\) and \(\psi_{x, y}\brk{y} = 1\).
Moreover, \(\psi_{x, y} = 0\) on \(\R \setminus \intvo{y^2/x}{x}\).
We compute then, under the change of variable \(z = y \brk{x/y}^t\),
\[
\begin{split}
 \int_0^\infty \norm{\psi_{x, y}' \brk{z}}^p z^{p - 1} \dd z
 &=\frac{1}{\abs{\ln \brk{x/y}}^p}
 \int_0^\infty \norm[\Big]{\psi' \brk[\Big]{\frac{\ln \brk{z/y}}{\ln \brk{x/y}}}}^p \frac{1}{z} \dd z\\
 &= \frac{1}{\abs{\ln \brk{x/y}}^{p - 1}}
 \int_0^\infty \abs{\psi'\brk{t}}\dd t.
\end{split}
\]
In view of the hypothesis \eqref{eq_ooshaiyohp6chahshuu9Ahre} on \(\omega\), we then have
\[
 1 = \abs{\psi_{x, y} \brk{x} - \psi_{x, y} \brk{y}}
 \le \omega \brk{x, y} \brk[\bigg]{\int_0^\infty \norm{\psi_{x, y}' \brk{z}}^p z^{p - 1} \dd z}^\frac{1}{p}
 \le C \frac{\omega \brk{x, y}}{\norm{\ln \brk{x/y}}^{\frac{p}{p- 1}}},
\]
from which the conclusion follows.
\end{proof}

\begin{proof}[Proof of \cref{theorem_optimal}]
In view of \cref{theorem_optimal_compact,theorem_optimal_compact_n_1} and the relationship between \(\Theta_{\alpha, \smash{\beta}, \gamma}\) and \(\Theta^0_{\smash{\beta}, \gamma}\) defined in \eqref{eq_hoh1az5ishahXeera3Daepod} and \eqref{eq_inoiX5Eiphu5saeQuu6so5ta} respectively, it suffices to consider the case \(-1 < \gamma \le p - n\).

We first assume that \(-1 <\gamma < p - n\).
We consider \(\psi_{y,R} \in C^\infty \brk{\R^n_+}\) defined by \eqref{eq_Paem4eezeisochaeRu3tei9a} in the proof of \cref{theorem_optimal_compact}. If \(R \ge \frac{y_n}{2}\), then \(\Ball^n_R\brk{y} \subset  \Ball^n_{3R}\brk{y^\prime, 0}\). Through an analysis of cases depending on the sign of \(\gamma\), it follows that
\begin{equation}
\label{eq_ooleeGhoodeiliu9rohpha2l}
\begin{split}
 \int_{\R^n_+} \abs{\Diff \psi_{y, R}  \brk{z}}^p {z_n}^\gamma \dd  z
& \le \frac{C}{R^p} \int_{\Ball^n_{3R}\brk{y^\prime, 0}} {z_n}^\gamma \dd z \\
& \le C R^{n + \gamma - p}.
\end{split}
\end{equation}
Combining \eqref{eq_quaino8pheiRieWoo0gu4xah} and \eqref{eq_ooleeGhoodeiliu9rohpha2l}, we obtain
\begin{equation}
\label{eq_uij9esiiZ1theepoo4eih1Aa}
 \int_{\R^n_+} \abs{\Diff \psi_{y, R}  \brk{z}}^p {z_n}^\gamma \dd  z
 \le C R^{n - p} \max \brk{y_n, R}^\gamma.
\end{equation}
By assumption \eqref{eq: hypothesis on omega} on \(\omega\), the inequality \eqref{eq_uij9esiiZ1theepoo4eih1Aa} then implies that
\begin{equation}
\label{eq_se9jie9unie1wah2oXaVo0ii}
\begin{split}
  1 = \abs{\psi_{y, \abs{x - y}} \brk{x} - \psi_{y, \abs{x - y}}\brk{y}}
  &\le
  \omega \brk{x, y}
  \brk[\bigg]{ \int_{\R^n_+} \abs{\Diff \psi_{y, \abs{x - y}} \brk{z}}^p z_n{}^\gamma \dd  z
  }^\frac{1}{p}\\
  &\le C
  \frac{\max \brk{y_n, \abs{x - y}}^{\smash{\frac{\gamma}{p}}}}
  {\abs{x - y}^{1 -\smash{\frac{n}{p}}}} \omega \brk{x, y}.
  \end{split}
\end{equation}
Interchanging the roles of \(x\) and \(y\), we thus get
\begin{equation}
	\label{eq: omega non compact}
\begin{split}
 \omega \brk{x, y} & \ge C  \abs{x - y}^{1 - \smash{\frac{n}{p}}}
 \max \brk{\max \brk{x_n, \abs{x - y}}^{\smash{-\frac{\gamma}{p}}}, \max \brk{y_n, \abs{x - y}}^{-\smash{\frac{\gamma}{p}}}}\\
 &\ge C \min \brk{1, 2^{-\smash{\frac{\gamma}{p}}}} \abs{x - y}^{1 - \smash{\frac{n}{p}}}\max \brk{x_n, y_n, \abs{x - y}}^{\smash{-\frac{\gamma}{p}}},
\end{split}
\end{equation}
since,
\begin{gather*}
	 \max\brk{x_n, \abs{x - y}} \le \max\brk{x_n, y_n, R} \le 2 \max \brk{x_n, \abs{x - y}},
	 \intertext{and}
	 \max\brk{y_n, \abs{x - y}} \le \max\brk{x_n, y_n, R} \le 2 \max \brk{y_n, \abs{x - y}}.
\end{gather*}
We deduce from \eqref{eq: omega non compact} and by definition \eqref{eq_hoh1az5ishahXeera3Daepod} of \(\Theta_{\alpha, \beta, \kappa}\), that, if \(-1 < \gamma < p - n\),
\[
 C \omega \brk{x, y}
 \ge \Theta_{\smash{1 - \frac{1}{p},1 - \frac{n}{p}, 1 - \frac{n + \gamma}{p}}}
	   \brk{x, y}.
\]

It remains to treat the case \(\gamma = p - n\).
In order to do this, we take a smooth function \(\theta \in C^\infty \brk{\R}\) such that \(\theta = 1\) on \(\intvl{-\infty}{0}\) and \(\theta = 0\) on \(\intvr{1}{+\infty}\). We then define, for all \(z \in \R^n_+\) and \(R > y_n\), the function
\[
 \theta_{y, R} \brk{z} =
 \theta \brk[\Big]{\frac{\ln \brk[\big]{\abs{z - y}/y_n}}{\ln \brk[\big]{R/y_n}}}.
\]
If \(R > y_n\),
then \(\abs{z - y}\ge y_n\) implies that \(z_n \le y_n + \abs{z - y} \le 2 \abs{z - y}\).
Since \(\gamma = p - n > 0\), we compute
\begin{equation}
	\label{eq: optimality gamma = p - n}
\begin{split}
  \int_{\R^n_+} \abs{\Diff \theta_{y, R} \brk{z}}^p {z_n}^\gamma \dd  z
 &\le \frac{C}{\brk[\big]{\ln \frac{R}{y_n}}^p} \int_{y_n \le \abs{z - y}\le R}
 \frac{{z_n}^\gamma}{\abs{z - y}^{p}} \dd z\\
 &\le \frac{C}{\brk[\big]{\ln \frac{R}{y_n}}^p} \int_{y_n \le \abs{z - y}\le R}
 \frac{1}{\abs{z - y}^{p - \gamma}} \dd z\\
 &\le \frac{C}{\brk[\big]{\ln \frac{R}{y_n}}^{p - 1}}\\
 & =  
 \frac{C}{\brk[\Big]{\ln \brk[\big]{\brk[\big]{\frac{R}{y_n}}^{\frac{p - n}{p - 1}}}}^{p - 1}},
\end{split}
\end{equation}
where the last equality follows from the identity \(\ln t = \frac{1}{\alpha} \ln\brk{t^\alpha}\).
In particular, we obtain from \eqref{eq: optimality gamma = p - n} that, if \(\norm{x - y} > y_n\),
\begin{equation}
\label{eq_eiyaiv8OhTh2koot9eiruP6i}
\begin{split}
  1 = \abs{\theta_{y, \abs{x - y}} \brk{x} - \theta_{y, \abs{x - y}}\brk{y}}
  &\le
  \omega \brk{x, y}
  \brk[\bigg]{ \int_{\R^n_+} \abs{\Diff \theta_{y, y_n} \brk{z}}^p z_n{}^\gamma \dd  z
  }^\frac{1}{p}\\
  &\le \frac{C}{\brk[\Big]{\ln \brk[\big]{\brk[\big]{\frac{\abs{x - y}}{ y_n}}^{\frac{p - n}{p - 1}}}}^{1 - \frac{1}{p}}}\omega \brk{x, y}.
\end{split}
\end{equation}
By definition \eqref{eq_hoh1az5ishahXeera3Daepod} of \(\Theta_{\alpha, \beta, \kappa}\), this implies that
\[
 C \omega \brk{x, y}
 \ge \Theta_{\smash{1 - \frac{1}{p},1 - \frac{n}{p}, 1 - \frac{n + \gamma}{p}}}
	   \brk{x, y},
\]
when \(\gamma = p - n\), concluding the proof.
\end{proof}
\begin{proposition}
\label{proposition_optimal_omega}
Let  \(n\geq 1\), \(\gamma \in \R\), and \(n < p < +\infty\).
There exists a function \(\omega_\ast \colon \R^n_+ \times \R^n_+ \to \intvr{0}{+\infty}\) such that, for every \(\omega \colon \R^n_+ \times \R^n_+ \to \intvr{0}{+\infty}\), the following are equivalent
\begin{enumerate}[label=(\roman*)]
\item \label{it_soo6ieshae5ku2woh0ohYa8i}
for every \(u \in  \smash{\dot{W}}^{1, p}_{\smash{\gamma}} \brk{\R^n_+} \cap C \brk{\R^n_+}\) and every \(x, y \in \R^n_+\),
\begin{equation}	 
	 \abs{u \brk{x} - u \brk{y}}
	 \le
	 \omega \brk{x, y}
	 \brk[\Big]{\int_{\R^n_+}
	 \norm{\Diff u\brk{z}}^p {z_n}^\gamma \dd z}^\frac{1}{p},
\label{eq_soo6ieshae5ku2woh0ohYa8i}
\end{equation}
\item 
\label{it_thie1ju7Ja8oveeJaetauv7V}
\(
\omega \ge \omega_\ast
\).
\end{enumerate}
Moreover,
\begin{enumerate}[label=(\roman*), resume]
 \item
 \label{it_cohrimiehoojeeVai2DahToh}
 \(\omega_\ast\) is a distance inducing the natural topology on \(\R^n_+\),
 \item
 \label{it_chae1aich9Koor9dee8saada}
 there exist constants \(c\), \(C \in \intvo{0}{+\infty}\) such that for every \(x\), \(y \in \R^n_+\)
 \[
 c\, \Theta_{\smash{1 - \frac{1}{p}, 1 - \frac{n}{p}, 1 - \frac{n + \gamma}{p}}} \brk{x, y}
 \le \omega_\ast \brk{x, y} \le C\, \Theta_{\smash{1 - \frac{1}{p}, 1 - \frac{n}{p}, 1 - \frac{n + \gamma}{p}}} \brk{x, y},
 \]
 \item
 \label{it_Oohaif6Io9Ieyeip9Thu0chi}
 for every \(t \in \intvo{0}{+\infty}\) and \(x\), \(y \in \R^n_+\),
 \[
  \omega_\ast \brk{t x, t y} = t^{1 - \frac{n + \gamma}{p}} \omega_\ast \brk{x, y},
 \]
 \item
 \label{it_jiena6hei0eiSahch1aimaes}
 for every \(h \in \R^{n - 1} \times \set{0}\) and \(x\), \(y \in \R^n_+\),
 \[
  \omega_\ast \brk{x + h, y + h} = \omega_\ast \brk{x, y}.
 \]
\end{enumerate}
\end{proposition}

\begin{proof}
We define \(\omega_\ast\) as the pointwise infimum of the functions \(\omega \colon \R^n_+ \times \R^n_+ \to \intvr{0}{+\infty}\) satisfying the condition \ref{it_soo6ieshae5ku2woh0ohYa8i}, ensuring the equivalence between \ref{it_soo6ieshae5ku2woh0ohYa8i} and \ref{it_thie1ju7Ja8oveeJaetauv7V}.
The assertion \ref{it_chae1aich9Koor9dee8saada} follows from \cref{thm: weighted Morrey--Sobolev} and \cref{theorem_optimal}.
The minimality property of \(\omega_\ast\) ensures that, for every \(x\), \(y \in \R^n_+\),
\[
 \omega_\ast\brk{x, y} = \omega_\ast \brk{y, x}
\]
and, for every \(x\), \(y\), \(z \in \R^n_+\),
\[
 \omega_\ast\brk{x, y} \le \omega_\ast \brk{x, z} + \omega_\ast \brk{z, x}.
\]
In view of \cref{thm: weighted Morrey--Sobolev}, \ref{it_cohrimiehoojeeVai2DahToh} holds.

The assertions \ref{it_Oohaif6Io9Ieyeip9Thu0chi} and \ref{it_jiena6hei0eiSahch1aimaes} follow from the minimality of \(\omega_\ast\) and applying the corresponding transformations to \(u\) in \ref{it_soo6ieshae5ku2woh0ohYa8i}.
\end{proof}
\Cref{proposition_optimal_omega} combined with standard separability argument shows that any function \(\omega \colon \R^n_+ \times \R^n_+ \to \intvr{0}{+\infty}\) such that, for every \(u \in  \smash{\dot{W}}^{1, p}_{\smash{\gamma}} \brk{\R^n_+} \cap C \brk{\R^n_+}\) and almost every \(x\), \(y \in \R^n_+\),
\begin{equation*}
\abs{u \brk{x} - u \brk{y}}\le
	 \omega \brk{x, y}
	 \brk[\Big]{\int_{\R^n_+}
	 \norm{\Diff u\brk{z}}^p {z_n}^\gamma \dd z}^\frac{1}{p},
\end{equation*}
satisfies \(\omega \ge \omega_\ast\) almost everywhere on \(\R^n_+\times \R^n_+\).

\begin{proposition}
	\label{proposition_optimal_omega_compact}
Let  \(n\geq 2\), \(\gamma \in \R\), and \(n < p < +\infty\).
There exists a function \(\omega_\ast \colon \R^n_+ \times \R^n_+ \to \intvr{0}{+\infty}\) such that, for every \(\omega \colon \R^n_+ \times \R^n_+ \to \intvr{0}{+\infty}\), the following are equivalent
\begin{enumerate}[label=(\roman*)]
\item for every \(u \in  C^\infty_c \brk{\R^n_+}\)
and \(x, y \in \R^n_+\),
\begin{equation}
\abs{u \brk{x} - u \brk{y}}
	 \le
	 \omega \brk{x, y}
	 \brk[\Big]{\int_{\R^n_+}
	 \norm{\Diff u\brk{z}}^p {z_n}^\gamma \dd z}^\frac{1}{p},
\end{equation}
\item 
\(
\omega \ge \omega_\ast.
\)
\end{enumerate}
Moreover,
\begin{enumerate}[label=(\roman*), resume]
 \item
 \(\omega_\ast\) is a distance inducing the natural topology on \(\R^n_+\),
 \item
 there exist constants \(c\), \(C \in \intvo{0}{+\infty}\) such that for every \(x\), \(y \in \R^n_+\)
 \[
 c\, \Theta^0_{\smash{ 1 - \frac{n}{p}, 1 - \frac{n + \gamma}{p}}} \brk{x, y}
 \le \omega_\ast \brk{x, y} \le C\, \Theta^0_{\smash{1 - \frac{n}{p}, 1 - \frac{n + \gamma}{p}}} \brk{x, y},
 \]
 \item
 for every \(t \in \intvo{0}{+\infty}\) and \(x\), \(y \in \R^n_+\),
 \[
  \omega_\ast \brk{t x, t y} = t^{1 - \frac{n + \gamma}{p}} \omega_\ast \brk{x, y},
 \]
 \item
 for every \(h \in \R^n\) and \(x\), \(y \in \R^n_+\),
 \[
  \omega_\ast \brk{x + h, y + h} = \omega_\ast \brk{x, y}.
 \]
\end{enumerate}
\end{proposition}

Again one can deduce from \cref{proposition_optimal_omega_compact} that any function \(\omega \colon \R^n_+ \times \R^n_+ \to \intvr{0}{+\infty}\) such that for every \(u \in  C^\infty_c \brk{\R^n_+}\) and almost every \(x\), \(y \in \R^n_+\),
\begin{equation*}
\abs{u \brk{x} - u \brk{y}}\le
	 \omega \brk{x, y}
	 \brk[\Big]{\int_{\R^n_+}
	 \norm{\Diff u\brk{z}}^p {z_n}^\gamma \dd z}^\frac{1}{p},
\end{equation*}
satisfies \(\omega \ge \omega_\ast\) almost everywhere on \(\R^n_+\times \R^n_+\).

\section{Hyperbolic Morrey--Sobolev inequality and optimality}
\label{section: hyperbolic Morrey--Sobolev}
The case \(\gamma = p - n\) corresponds to the Sobolev--Morrey inequality on the \(n\)-dimensional hyperbolic space \(\Hset^n\) in the Poincaré half-space model,
whose distance is given by
\begin{equation}
 \mathrm{d}_{\Hset^n} \brk{x, y}
 =
 2 \operatorname{sinh}^{-1} \brk[\bigg]{\frac{\abs{x - y}}{2 \sqrt{x_n y_n}}}.
\end{equation}

\begin{lemma}
\label{lemma_conf_square}
Let \(n \geq 1\). For all \(x\), \(y \in \Hset^n\), one has
\[
  \frac{\abs{x - y}}{\sqrt{x_n y_n}}\le
 \frac{\abs{x - y}}{\min \brk{x_n,y_n}}
 \le \max \brk[\Big]{\sqrt{2} \frac{\abs{x - y}}{\sqrt{x_n y_n}},
  2 \frac{\abs{x - y}^2}{x_n y_n}}.
\]
\end{lemma}

\begin{proof}
We assume, without loss of generality, that \(x_n \le y_n\).
We then have
\begin{equation*}
 \frac{\abs{x - y}}{\sqrt{x_n y_n}}\le
 \frac{\abs{x - y}}{x_n},
\end{equation*}
which is the first inequality.

For the second inequality, we either have \(\abs{x - y}\le x_n\) and then
 \(
   y_n \le x_n + \abs{x - y} \le 2 x_n
 \),
 so that
 \[
 \frac{\abs{x - y}}{x_n} \le \sqrt{2} \frac{\abs{x - y}}{\sqrt{x_n y_n}},
 \]
 or we have \(\abs{x - y}\ge  x_n\),
 and therefore \(y_n \le x_n + \abs{x - y} \le 2 \abs{x - y}\), which then implies that
 \[
  \frac{\abs{x - y}}{x_n}
  \le 2 \frac{\abs{x - y}}{x_n} \frac{\abs{x - y}}{y_n} = 2 \frac{\abs{x - y}^2}{x_n y_n}. \qedhere
 \]
\end{proof}

\begin{proof}[Proof of \cref{theorem_hyperbolic}]
Thanks to \cref{lemma_conf_square} and noting that
\[
 \ln \brk{1 + t}  \le \sinh^{-1} \brk{t} = \ln \brk{t + \sqrt{1 + t^2}} \le \ln \brk{1 + 2t},
\]
we get
\[
 c \mathrm{d}_{\Hset^n}\brk{x, y}
 \le \ln \brk[\Big]{1 + \frac{\abs{x - y}}{\min \brk{x_n, y_n}}}
 \le
 C \mathrm{d}_{\Hset^n} \brk{x, y},
\]
and thus
\begin{multline*}
c \max\brk{\mathrm{d}_{\Hset^n} \brk{x, y}^{\frac{p - n}{p - 1}}, \mathrm{d}_{\Hset^n} \brk{x, y}}\\
\le
  \ln \brk[\Big]{1 + \brk[\Big]{\frac{\abs{x - y}}{\min \brk{x_n, y_n}}}^\frac{p - n}{p - 1}}\\
  \le
  C \max\brk{\mathrm{d}_{\Hset^n} \brk{x, y}^{\frac{p - n}{p - 1}}, \mathrm{d}_{\Hset^n} \brk{x, y}}.
\end{multline*}
The conclusion then follows from \cref{thm: weighted Morrey--Sobolev,theorem_optimal}.
\end{proof}
\begin{proof}[Proof of \cref{theorem_hyperbolic_compact}]
We proceed as in the proof of \cref{theorem_hyperbolic} to get 
\begin{equation*}
c \max\brk{\mathrm{d}_{\Hset^n} \brk{x, y},1}
\le
 \frac{\norm{x - y}}{\max \brk{\min \brk{x_n, y_n}, \norm{x - y}}}
  \le
  C \max\brk{\mathrm{d}_{\Hset^n} \brk{x, y}, 1}.
\end{equation*}
Noting that 
\[\max \brk{\min \brk{x_n, y_n}, \norm{x - y}}
\le 
 \max \brk{x_n, y_n, \norm{x - y}}
 \le 
 2\max \brk{\min \brk{x_n, y_n}, \norm{x - y}},
\]
the conclusion then follows from \cref{thm: weighted Morrey--Sobolev compactly supported} and \cref{theorem_optimal_compact}.
\end{proof}

\end{document}